\documentclass[pdflatex,sn-mathphys,Numbered]{sn-jnl}

\usepackage{graphicx}%
\usepackage{amsmath,amssymb,amsfonts}%
\usepackage{amsthm}%
\usepackage{mathrsfs}%
\usepackage[title]{appendix}%
\usepackage{xcolor}%
\usepackage{textcomp}%
\usepackage{manyfoot}%
\usepackage{booktabs}%
\usepackage{algorithm}%
\usepackage{algorithmicx}%
\usepackage{algpseudocode}%
\usepackage{listings}%

\usepackage{braket,amsfonts,amsopn,amsmath,yhmath}
\allowdisplaybreaks[4]
\usepackage{amscd}
\usepackage{cleveref}
\usepackage{bigstrut,multirow}
\usepackage{setspace}
\usepackage{subfigure}
\usepackage{booktabs}

\usepackage{graphicx}
\usepackage{enumitem}  

\DeclareMathOperator{\Graph}{Graph}
\newcommand{\innerP}[2]{\left\langle {#1},{#2} \right\rangle}
\newcommand{\dist}{\mathrm{dist}}

\newcommand{\dom}{\mathrm{dom}}
\newcommand{\epi}{\mathrm{epi}}
\newcommand{\prox}{\mathrm{prox}}

\newcommand{\Liminf}{\mathop{\lim\inf}}

\newcommand{\One}{\mathbf{1}}

\newcommand{\Prob}[1]{\mathbb{P}\left[ #1 \right]}

\newcommand{\StatementNum}[1]{{\upshape(\romannumeral #1 )}}
\newcommand{\Vxy}{\begin{bmatrix} x\\y \end{bmatrix}}

\newtheorem{assumption}{Assumption}

\theoremstyle{thmstyleone}%
\newtheorem{theorem}{Theorem}[section]
\newtheorem{lemma}[theorem]{Lemma}
\newtheorem{corollary}[theorem]{Corollary}
\newtheorem{proposition}[theorem]{Proposition}%

\theoremstyle{thmstyletwo}%
\newtheorem{remark}{Remark}%

\theoremstyle{thmstylethree}%
\newtheorem{definition}{Definition}%

\begin{document}

\title{An equivalent reformulation and multi-proximity gradient algorithms for a class of nonsmooth fractional programming}


\author[1]{\fnm{Junpeng} \sur{Zhou}}\email{zhoujp5@mail2.sysu.edu.cn}
\equalcont{These authors contributed equally to this work.}

\author[2]{\fnm{Na} \sur{Zhang}}\email{nzhsysu@gmail.com}
\equalcont{These authors contributed equally to this work.}

\author*[3]{\fnm{Qia} \sur{Li}}\email{liqia@mail.sysu.edu.cn}

\affil*[1]{\orgdiv{School of Computer Sciences and Engineering}, \orgname{Sun Yat-sen University}, \orgaddress{ \city{Guangzhou}, \postcode{510006}, \country{China}}}

\affil[2]{\orgdiv{Department of Applied Mathematics, College of Mathematics and Informatics}, \orgname{South China Agricultural University}, \orgaddress{ \city{Guangzhou}, \postcode{510642},  \country{China}}}

\affil[3]{\orgdiv{School of Computer Sciences and Engineering, Guangdong Province Key Laboratory of Computational Science}, \orgname{Sun Yat-sen University}, \orgaddress{ \city{Guangzhou}, \postcode{510006},  \country{China}}}



\abstract{In this paper, we consider a class of structured fractional programs, where the numerator part is the sum of a block-separable (possibly nonsmooth nonconvex) function and a locally Lipschitz differentiable (possibly nonconvex) function, while the denominator is a convex (possibly nonsmooth) function. We first present a novel reformulation for the original problem and show the relationship between optimal solutions, critical points and KL exponents of these two problems. Inspired by the reformulation, we propose a flexible framework of multi-proximity gradient algorithms (MPGA), which computes the proximity operator with respect to the Fenchel conjugate associated with the convex denominator of the original problem rather than evaluating its subgradient as in the existing methods. Also, MPGA employs a nonmonotone linear-search scheme in its gradient descent step, since the smooth part in the numerator of the original problem is not globally Lipschitz differentiable. Based on the framework of MPGA, we develop two specific algorithms, namely, cyclic MPGA and randomized MPGA, and establish their subsequential convergence under mild conditions. Moreover, the sequential convergence of cyclic MPGA with the monotone line-search (CMPGA\_ML) is guaranteed if the extended objective associated with the reformulated problem satisfies the Kurdyka-{\L}ojasiewicz (KL) property and some other mild assumptions. In particular, we prove that the corresponding KL exponents are $\frac{1}{2}$ for several special cases of the fractional programs, and so, CMPGA\_ML exhibits a linear convergence rate. Finally, some preliminary numerical experiments are performed to demonstrate the efficiency of our proposed algorithms.}

\keywords{fractional programming, equivalent reformulation, proximal algorithm, KL exponent}


\pacs[MSC Classification]{90C26, 90C30, 65K05}

\maketitle

\section{Introduction}\label{section:introduction}
In this paper, we consider a class of single-ratio fractional optimization problem
\begin{equation}\label{problem:root}
	\min\; \left\{\frac{f(x)+h(x)}{g(x)}:x\in\Omega \right\},
\end{equation}
where $f,g,h: \mathbb{R}^{n}\to (-\infty,+\infty]$ are proper lower semicontinuous functions and the set $\Omega := \{x\in\mathbb{R}^n:g(x)\neq 0 \}$ is nonempty. Moreover, we assume that $f$ is a block-separable function
\begin{equation*}
	f(x) = \sum_{i=1}^{N} f_i(x_i),
\end{equation*}
where each $x_i$ is a subvector of $x$ with dimension $n_i$, $\{x_i:i=1,2,...,N \} $ forms a partition of the components of $x$, and each $f_i:\mathbb{R}^{n_i}\to (-\infty,+\infty]$ is a proper lower semicontinuous function. Throughout the paper, we make the following blanket assumptions on problem \eqref{problem:root}.
\begin{assumption}\label{assumption0}
	\indent
	\begin{enumerate}[label = {\upshape(\roman*)}]
		\item $f$ is continuous on $\dom(f)$ and each $f_i$ is bounded below;
		\item $h$ is locally Lipschitz differentiable on $\Omega$, i.e., for any $x\in\Omega$, there exist a neighborhood $\mathcal{B}(x)\subseteq\mathbb{R}^n$ and a constant $L_x$ such that $\|\nabla h(u)-\nabla h(v)\|_2\leq L_x\|u-v\|_2$ holds for any $u,v\in\mathcal{B}(x)\cap \Omega$.
		\item $g$ is convex, real-valued and non-negative on $\mathbb{R}^n$.
		\item $f+h$ is non-negative on $\dom(f)$.
	\end{enumerate}
	
\end{assumption}

This class of fractional optimization problems encompasses many important optimization models arising from diverse areas. In this paper, we are particularly interested in the scale invariant sparse signal recovery models, which has recently received considerable attentions \cite{Li-Shen-Zhang-Zhou:2022ACHA,Rahimi-Wang-Dong-Lou:2019SIAM-SC,TaoMin:2022SIAM-SC,Wang-Tao-Nagy-Lou:2021SIAM-ImageScience,Wang-Yan-Rahimi-Lou:2020IEEE,Yin-Esser-Xin:2014CIS, Zeng-Yu-Pong:2021SIAM-OPT}. For example, the $L_1/L_2$ (the quotient of the $\ell_1$ and $\ell_2$ norms) least square model is in the form of
\begin{equation}\label{problem:L1dL2}
	\min\left\{\frac{\|x\|_1}{\|x\|_2}+\frac{\lambda}{2}\|Ax-b\|^2_2:\underline{x}\leq x \leq \overline{x}, x\neq 0, x\in\mathbb{R}^n  \right\},
\end{equation}
where $A\in\mathbb{R}^{m\times n}$, $b\in\mathbb{R}^m\backslash\{0 \}$, $\underline{x},\overline{x}\in\mathbb{R}^n$ and $\lambda >0$. Problem \eqref{problem:L1dL2} is a special case of problem \eqref{problem:root} when $h(x)= \frac{\lambda}{2}\|x\|_2\|Ax - b\|_2^2$, $g(x)=\|x\|_2$, $f_i$ is the sum of $\ell_1$ norm and the indicator function on $\{x_i\in\mathbb{R}^{n_i}: \underline{x}_i\leq x_i \leq \overline{x}_i  \}$ for $i = 1,2,...,N$. One can easily check that problem \eqref{problem:L1dL2} satisfies \Cref{assumption0}. As another example, we refer to the $L_1/S_K$ (the quotient of $\ell_1$ norm and the vector $K$-norm) sparse signal recovery model \cite{Li-Shen-Zhang-Zhou:2022ACHA} which can be formulated into
\begin{equation}\label{problem:L1dSK}
	\min\left\{ \frac{\|x\|_1+\frac{\lambda}{2}\|Ax-b\|_2^2}{\|x\|_{(K)}} : \underline{x}\leq x\leq \overline{x},x\neq 0, x\in\mathbb{R}^n \right\}.
\end{equation}
In the above problem, $\|x\|_{(K)}$ is the vector $K$-norm \cite{Wu-Ding-Sun-Toh:2014SIAMOPT, Gotoh-Takeda-Tono:2018Mathematical_Programming} of $x$, which is defined as the sum of the $K$ largest absolute values of entries in $x$. Clearly, problem \eqref{problem:root} reduces to problem \eqref{problem:L1dSK} when $h(x)=\frac{\lambda}{2}\|Ax -b\|^2_2$, $g(x) = \|x \|_{(K)}$ and $f_i$ is chosen the same as those in the prior example. We can also easily verify that problem \eqref{problem:L1dSK} satisfies \Cref{assumption0}.

For tackling single-ratio fractional optimization problems, one popular class of approaches is the Dinkelbach's method and its variants \cite{Crouzeix-Ferland-Schaible:1985JOTA,Dinkelbach-Werner:MS1967,Ibaraki:1983Mathematical_Programming,Schaible:1976Fractional_Programming}. Given an iterate $x^k$, this class of approaches for problem \eqref{problem:root} typically generates the next iterate by solving the following optimization problem
\begin{equation}\label{eq: z0725 1722}
	\min\{f(x)+h(x)-c_k g(x) :x\in\mathbb{R}^n \},
\end{equation}
where $c_k$ is renewed with $c_k:=\frac{f(x^k)+h(x^k)}{g(x^k)}$.  In general, it is very difficult and expensive to directly address \eqref{eq: z0725 1722} due to the possible non-convexity and non-smoothness. To remedy this issue, when $h$ is globally Lipschitz differentiable, several first-order algorithms \cite{Bot-Dao-Li:2021MathematicsofOperationsResearch,Li-Shen-Zhang-Zhou:2022ACHA,NaZhang-QiaLi:2022SIAM-OPT} are developed for solving problem \eqref{problem:root} by processing $f$, $h$ and $g$ separately. More precisely, in each iteration all these algorithms perform a gradient step with respect to $h$, a subgradient step with respect to $g$ and a proximal step with respect to $f$. The subproblems involved are usually much easier to handle and sometimes even have closed-form solutions, which leads to the high efficiency of these algorithms.
Subsequential convergence of the solution sequence generated by them is established, while sequential convergence is proved under further assumptions which includes the Kurdyka-{\L}ojasiewicz (KL) property of some auxiliary functions.

Differently from the aforementioned work \cite{Bot-Dao-Li:2021MathematicsofOperationsResearch, Li-Shen-Zhang-Zhou:2022ACHA,NaZhang-QiaLi:2022SIAM-OPT}, this paper presents a novel reformulation for problem \eqref{problem:root}, and then develops multi-proximity  algorithms based on this reformulated problem, which is in the form of
\begin{equation}\label{problem:root primal-dual}
	\min\; \left\{\frac{f(x)+h(x)}{\innerP{x}{y}-g^*(y)}:
	(x,y)\in\mathbb{R}^n\times \mathbb{R}^n, \innerP{x}{y}-g^*(y)>0
	\right\}.
\end{equation}
Here, $g^*:\mathbb{R}^n\to(-\infty,+\infty]$ denotes the classical Fenchel conjugate function of $g$. Specifically, the contributions of this paper are summarized as follow:
\begin{enumerate}[leftmargin=0.5cm, itemindent=0cm]
	\item We give an equivalent reformulation \eqref{problem:root primal-dual} for problem \eqref{problem:root} in the sense that both problems have the same optimal value, and $x^{\star}\in\mathbb{R}^n$ is an optimal solution of problem \eqref{problem:root} if and only if $(x^{\star},y^{\star}) $ is an optimal solution of problem \eqref{problem:root primal-dual} for some $y^{\star}\in\mathbb{R}^n $. Moreover, when $g^*$ is continuous on its domain, we show that $x^{\star}$ is a critical point of problem \eqref{problem:root} if and only if $(x^{\star},y^{\star}) $ is a stationary point of problem \eqref{problem:root primal-dual} for some $y^{\star}$. Under further mild conditions, we prove that if the extended objective of problem \eqref{problem:root primal-dual} is a KL function with a certain KL exponent, then the extended objective of problem \eqref{problem:root} is also a KL function with the same KL exponent.
	\item By exploiting the structure of the numerator and denominator, we propose a general framework of multi-proximity gradient algorithms for solving problem \eqref{problem:root primal-dual}. At each iteration, this method first picks a block from $\{x_i: i=1,2,...,N\}$ and $y$, and then solves typically a proximal subproblem associated with the chosen block while fixing the remaining blocks at their last renewed values. Moreover, since $h$ is not globally Lipschitz differentiable, a nonmonotone line-search scheme is incorporated to determine the step size of gradient descent if the block $y$ is not selected. We prove that the MPGA always owns several favorable properties, regardless of the order for updating the blocks.
	\item We investigate the convergence of the MPGA in the case where the update block is chosen cyclically (CMPGA) or randomly (RMPGA). We show that any accumulation point of the solution sequence generated by CMPGA is a critical point of problem \eqref{problem:root}, while that by RMPGA is a critical point almost surely. In addition, we establish the convergence of the whole solution sequence generated by CMPGA with monotone line search (CMPGA\_ML), by further assuming that $f$ is locally Lipschitz continuous, $g^*$ satisfies the calmness condition on its domain, and the extended objective of problem \eqref{problem:root primal-dual} satisfies the KL property. The convergence rate of CMPGA\_ML is also estimated based on the KL exponent of the extended objective of problem \eqref{problem:root primal-dual}. Specifically, we show that, for several special cases of problem \eqref{problem:root}, the KL exponents associated with their reformulated problem \eqref{problem:root primal-dual} are $\frac{1}{2}$, and thus CMPGA\_ML exhibits a linear convergence rate.
\end{enumerate}

The remaining part of this paper is organized as follows. In \Cref{section:Notation_and_preliminaries}, we introduce notation and some preliminaries. \Cref{section: Connections between primal and primal-dual problems} is devoted to a study of connections between problem \eqref{problem:root} and \eqref{problem:root primal-dual}. We propose the multi-proximity gradient algorithm in \Cref{section: MPGA} and analyze its subsequential convergence in Section \ref{section: Subsequential convergence analysis}. In \Cref{section: Global convergence}, the sequential convergence and convergence rate of CMPGA\_ML is established. We present in \Cref{section: Numerical experiments} some numerical results for the $L_1/L_2$ and $L_1/S_K$ signal recovery problems to demonstrate the efficiency of our proposed algorithms. 

\section{Notation and preliminaries}\label{section:Notation_and_preliminaries}

We begin with our preferred notations. We denote the Euclidean space of dimension $n$ and the set of nonnegative integers by $\mathbb{R}^n$ and $\mathbb{N}$. Let $\mathbb{N}_M:= \{1,2,...,M \}$, $\mathbb{N}_M^0 := \{0 \}\cup \mathbb{N}_M$ for a positive integer $M$, and $[x]_+ :=\max \{0,x\}$. For a vector $x\in\mathbb{R}^n$, we use $\mathcal{B}(x)$ to denote an open neighborhood of $x$ in $\mathbb{R}^n$, and use $\mathcal{B}(x,\delta)$ to denote an open ball in $\mathbb{R}^n$ with the center $x$ and the radius $\delta>0$. Given a function $\varphi:\mathbb{R}^n\to(-\infty,+\infty]$ being finite at $x$ and a positive number $\epsilon>0$, we use $\mathcal{B}_{\varphi}^{\epsilon}(x,\delta)$ to denote the set $\mathcal{B}(x,\delta)\cap \{z\in\mathbb{R}^n:\varphi(x)<\varphi(z)<\varphi(x)+\epsilon \}$. Moreover, for a vector $x\in\mathbb{R}^n$, $x_i$ always denotes a subvector of $x$ with cardinality $n_i$ and $\{x_i:i\in\mathbb{N}_N \} $ forms a partition of the components of $x$. We use $\nabla_i h$ to denote the partial gradient of $h$ with respect to $x_i$, $i\in\mathbb{N}_N$. The $\ell_1$ norm, the $\ell_2$ norm and the inner-product of $\mathbb{R}^n$ are denoted by $\|\cdot\|_1,~\|\cdot\|_2$ and $\innerP{\cdot}{\cdot}$, respectively. We use $x\in \mathcal{A}\to x^{\star}$ to denote that the variant $x$ converges to $x^{\star}$ within the set $\mathcal{A}$ and use $\{x^k:k\in\mathcal{K}\}\to x^{\star}$ to denote that the sequence $x^k$ indexed by $\mathcal{K}$ converges to $x^{\star}$, i.e., $\lim_{k\in\mathcal{K}\to\infty} x^k = x^{\star}$. The indicator function on a nonempty set $\mathcal{A}\subseteq\mathbb{R}^n$ is defined by
\begin{equation*}
	\iota_{\mathcal{A}}(x) = \begin{cases}
		0, & \text{if~~}x\in\mathcal{A},\\
		+\infty, &\text{else}.
	\end{cases}
\end{equation*}
For two sets $\mathcal{A}_1$ and $\mathcal{A}_2$, $\mathcal{A}_1 \times \mathcal{A}_2$ denotes the Cartesian product of $\mathcal{A}_1$ and $\mathcal{A}_2$. The distance from a point $x\in\mathbb{R}^n$ to a set $\mathcal{A}\subseteq \mathbb{R}^n$ is denoted by $\dist(x,\mathcal{A}) := \inf\{\|x-y\|_2:y\in\mathcal{A} \}$, and we adopt $\dist(x,\emptyset) = +\infty$.

In the remaining part of this section, we present some preliminaries on some generalized subdifferentials of nonconvex functions \cite{Boris:2006Variational_analysis, Rockafellar2004Variational} and the KL property \cite{Attouch-bolt-redont-soubeyran:2010}.


\subsection{Generalized subdifferentials}

An extended-real-value function $\varphi:\mathbb{R}^n\to (-\infty,+\infty]$ is said to be proper if its domain $\dom(\varphi) := \{x\in\mathbb{R}^n: \varphi(x)<+\infty \}$ is nonempty. A proper function $\varphi$ is said to be closed if $\varphi$ is lower semicontinuous on $\mathbb{R}^n$. For a proper function $\varphi$, its Fr{\'e}chet and limiting subdifferential at $x\in\dom(\varphi)$ are defined respectively by
\begin{align*}
	\widehat{\partial} \varphi (x)&:=\left\{ y\in\mathbb{R}^n:\mathop{\lim\inf}\limits_{\substack{z\to x\\z\neq x}}\;
	\frac{\varphi(z)-\varphi(x)-\innerP{y}{z-x}}{\|z-x\|_2}\geq 0
	\right\},\\
	{\partial}\varphi(x) &:= \left\{ y\in\mathbb{R}^n:\exists x^k\to x,~\varphi(x^k)\to\varphi(x),~ y^k\in\widehat{\partial}\varphi(x^k) \text{ with } y^k\to y \right\}.
\end{align*}
We define $\dom(\partial \varphi):= \{x\in\dom(\varphi):\partial \varphi(x)\neq \emptyset \}$. A vector $x^{\star}\in\dom(\varphi)$ is said to be a stationary point of $\varphi$ if $0\in \widehat{\partial}\varphi (x^{\star})$. It is straightforward to verify that $\widehat{\partial}\varphi(x)\subseteq {\partial}\varphi(x)$, $\widehat{\partial}(\alpha \varphi)(x) = \alpha \widehat{\partial}\varphi(x)$ and $\partial(\alpha \varphi)(x) = \alpha \partial\varphi(x)$ hold for any $x\in\dom(\varphi)$ and $\alpha >0$. If $\varphi$ is convex, Fr{\'e}chet subdifferential and limiting subdifferential coincide with the classical subdifferential at any $x\in\dom(\varphi)$ (\cite[Proposition 8.12]{Rockafellar2004Variational}), i.e.,
\begin{equation*}
	\widehat{\partial} \varphi(x) = \partial\varphi(x) = \{y\in\mathbb{R}^n:\varphi(z)-\varphi(x)-\innerP{y}{z-x}\geq0,\forall z\in\mathbb{R}^n \}.
\end{equation*}
It is known that $\widehat{\partial} \varphi (x) = \{\nabla\varphi(x) \}$ if $\varphi$ is differentiable at $x$. We say $\varphi$ is continuously differentiable at $x$, if $\varphi$ is differentiable on some $\mathcal{B}(x)$ and $\nabla\varphi$ is continuous at $x$. It can be verify that $\partial\varphi(x) =\{\nabla\varphi(x) \}$ holds if $\varphi$ is continuously differentiable at $x$. In addition, there are some useful calculus results on Fr{\'e}chet subdifferential. Let proper functions $\varphi_1,~\varphi_2:\mathbb{R}^n\to(-\infty,+\infty]$ be closed, and $x\in\dom(\varphi_1)\cap\dom(\varphi_2)$, then $\widehat{\partial}\varphi_1(x) + \widehat{\partial}\varphi_2(x)\subseteq\widehat{\partial}(\varphi_1+\varphi_2)(x)$ (\cite[Corollary 10.9]{Rockafellar2004Variational}), where the equality holds if $\varphi_1$ or $\varphi_2$ is differentiable at $x$ (\cite[Exercise 8.8(c)]{Rockafellar2004Variational}). Note that, $\partial(\varphi_1+\varphi_2)(x)\subseteq \partial\varphi_1(x) + \partial\varphi_2(x)$ holds when $\varphi_1$ or $\varphi_2$ is locally Lipschitz continuous at $x$ (\cite[Exercise 10.10]{Rockafellar2004Variational}), and  holds with equality when $\varphi_1$ or $\varphi_2$ is continuously differentiable at $x$ (\cite[Exercise 10.10 and Theorem 9.13(a)(c)]{Rockafellar2004Variational}). Let proper functions $\varphi_i:\mathbb{R}^{n_i}\to(-\infty,+\infty]$, $i\in\mathbb{N}_N$, be closed, then for the block-separable function $\varphi(x) = \sum_{i=1}^{N} \varphi_i(x_i)$, there hold $\widehat{\partial}\varphi(x) = \widehat{\partial}\varphi_1(x_1)\times \widehat{\partial}\varphi_2(x_2)\times ... \times \widehat{\partial}\varphi_N(x_N)$ and $\partial\varphi(x) = \partial\varphi_1(x_1)\times \partial\varphi_2(x_2)\times ... \times \partial\varphi_N(x_N)$ at any $x=(x_1,x_2,...,x_N)\in \dom(\varphi_1)\times ... \times \dom(\varphi_N)$ (\cite[Proposition 10.5]{Rockafellar2004Variational}).

Next we review quotient rules for the Fr{\'e}chet subdifferential of $\varphi_1/\varphi_2$.
To this end, we first assume $\dom(\varphi_2) = \mathbb{R}^n$ and introduce two functions related to the ratio of $\varphi_1$ and $\varphi_2$. The first one is $\tau:\mathbb{R}^n\to (-\infty,+\infty]$ defined at $x\in\mathbb{R}^n$ as
\begin{equation*}
	\tau(x) := \begin{cases}
		\frac{\varphi_1(x)}{\varphi_2(x)}, &\text{ if $x\in\dom(\varphi_1)$ and $\varphi_2(x)\neq 0$,}\\
		+\infty, &\text{else.}\end{cases}
\end{equation*}
The second one is $\rho:\mathbb{R}^n\times\mathbb{R}^n\to (-\infty,+\infty]$ defined at $(x,y)\in\mathbb{R}^n\times\mathbb{R}^n$ as
\begin{equation*}
	\rho(x,y) := \begin{cases}
		\frac{\varphi_1(x)}{\innerP{x}{y}-\varphi_2^*(y)}, &\text{if $(x,y)\in\dom(\varphi_1)\times\dom(\varphi_2^*)$ and $\innerP{x}{y}-\varphi_2^*(y)>0$},\\
		+\infty, &\text{else,}
	\end{cases}
\end{equation*}
where $\varphi_2^*$ is the convex conjugate of $\varphi_2$, i.e., $\varphi_2^*(y) := \sup\{\innerP{x}{y}-\varphi_2(y): x\in\mathbb{R}^n \}$.
We also need the concept of calmness condition \cite[Section 8.F]{Rockafellar2004Variational}. The function $\varphi:\mathbb{R}\to(-\infty,+\infty]$ is said to satisfy the calmness condition at $x\in\dom(\varphi)$ relative to $\mathcal{A}\subseteq \mathbb{R}^n$, if there exists $\kappa_x>0$ and a neighborhood $\mathcal{B}(x)$ of $x$, such that $|\varphi(u)-\varphi(x)| \leq \kappa_x \|u-x\|_2$ holds for any $u\in\mathcal{B}(x)\cap\mathcal{A}$. We say $\varphi$ satisfies the calmness condition on $\mathcal{A}$, if $\varphi$ satisfies the calmness condition at each point in $\mathcal{A}$ relative to $\mathcal{A}$. The following proposition concerns the Fr{\'e}chet subdifferentials of $\tau$ and $\rho$.

\begin{proposition}\label{ppsition:2.2}
	Let $(x,y)\in\dom(\varphi_1)\times\dom(\varphi_2^*)$ and $a_1=\varphi_1(x)$, $a_2=\varphi_2(x)$, $a_3=\innerP{x}{y}-\varphi_2^*(y)$. Suppose that $\varphi_1$ is continuous at $x$ relative to $\dom(\varphi_1)$, then the following two statements hold.
	\begin{enumerate}[label = {\upshape(\roman*)}]
		\item \cite[Proposition 2.2]{NaZhang-QiaLi:2022SIAM-OPT} If $a_2>0$ and $\varphi_2$ satisfies the calmness condition at $x$, then
		\begin{equation*}
			\widehat{\partial}\tau (x) = \frac{\widehat{\partial}(a_2\varphi_1-a_1\varphi_2)(x)}{a^2_2};
		\end{equation*}
		\item \cite[Proposition 2.3]{Li-Shen-Zhang-Zhou:2022ACHA} If $a_1>0$, $a_3>0$ and $\varphi_2^*$ satisfies the calmness condition at $y$ relative to $\dom(\varphi_2^*)$, then
		\begin{equation*}
			\widehat{\partial}\rho(x,y) = \frac{a_3\widehat{\partial}\varphi_1(x)-a_1y}{a^2_3}
			\times
			\frac{\widehat{\partial}(a_1\varphi_2^*)(y)-a_1x}{a^2_3}.
		\end{equation*}
	\end{enumerate}
\end{proposition}

The following proposition is about the limiting subdifferentials of $\tau$ and $\rho$.


\begin{proposition}\label{Corollary: limiting subdiff of frac}
	Suppose that $\varphi_1$ is closed and continuous on $\dom(\varphi_1)$. 
	\begin{enumerate}[label = {\upshape(\roman*)}]
		\item Let $x\in\dom(\tau)$ with $a_1 = \varphi_1(x)>0$ and  $a_2=\varphi_2(x)>0$. If $\varphi_2$ is locally Lipschitz continuous around $x$ and such that $\widehat{\partial}\varphi_2(x)$ is nonempty around this point,\footnote{These assumptions on $\varphi_2$ automatically hold when $\varphi_2 $ is convex and continuous around $x$.} then
		\begin{equation}\label{eq: subdiff of tau}
			\partial \tau(x) \subseteq \frac{1}{a_2}\left (\partial \varphi_1(x)-\frac{a_1}{a_2}\partial\varphi_2(x)\right ),
		\end{equation}
		and hence $\dom(\partial \tau) \subseteq \dom(\partial \varphi_1)$. Furthermore, the relation \eqref{eq: subdiff of tau} becomes an equality if $\varphi_2$ is continuously differentiable at $x$.
		\item Let $(x,y)\in\dom(\rho)$ with $a_1 =\varphi_1 (x)>0$ and $a_3 =\innerP{x}{y}-\varphi_2^*(y) > 0$. If $\varphi_2^*$ satisfies the calmness condition around $x$ relative to $\dom(\varphi_2^*)$, then
		\begin{equation*}
			\partial\rho(x,y) = \frac{1}{a_3}\left ( 
			\left (\partial\varphi_1(x)-\frac{a_1}{a_3}y \right)
			\times
			\frac{a_1}{a_3}(\partial\varphi_2^*(y)-x)
			\right ),
		\end{equation*}
		and hence $\dom(\partial \rho) = \dom(\partial \varphi_1)\times\dom(\partial \varphi_2^*)$.
	\end{enumerate}
\end{proposition}
\begin{proof}
	Due to the continuity of $\varphi_1$ on $\dom(\varphi_1)$ and the Lipschitz continuity of $\varphi_2$ around $x$, there exists some neighborhood $\mathcal{B}(x)$ of $x$ such that $\varphi_1>0$ and $\varphi_2>0$ holds on $\mathcal{B}(x)\cap \dom(\tau)$, while $\widehat{\partial}\varphi_2$ is uniformly bounded on $\mathcal{B}(x)$ (see \cite[Theorem 3.52]{Boris:2006Variational_analysis}). Let $w\in\partial \tau(x)$. By the definition of limiting subdifferentials, there exists some $x^k\in\mathcal{B}(x)\cap\dom(\tau)\to x$ with $\tau(x^k)\to\tau(x)$, $w^k\in\widehat{\partial}\tau(x^k)$ and $w^k \to w$. Owing to $\varphi_1(x^k)>0$ and $\varphi_2(x^k)>0$, \Cref{ppsition:2.2} \StatementNum{1} indicates that
	\begin{equation}\label{eq: z0901 2330}
		\widehat{\partial}\tau (x^k) = \frac{\widehat{\partial}(\varphi_1-\tau(x^k)\varphi_2)(x^k)}{\varphi_2(x^k)}\subseteq \frac{1}{\varphi_2(x^k)}\Big (\widehat{\partial} \varphi_1(x^k)-\tau(x^k)\widehat{\partial}\varphi_2(x^k) \Big ),
	\end{equation}
	where the last relation is deduced by invoking \cite[Theorem 3.1]{Mordukhovich:2006Frechet_subdifferential} and noting that $\widehat{\partial} \varphi_2(x)$ is nonempty for the real-valued function $\varphi$.	 
	It follows from \eqref{eq: z0901 2330} that each $w^k\in \widehat{\partial}\tau (x^k)$ can be represented by 
	$w^k = \left(v_1^k - \tau(x^k)v^k_2 \right) / \varphi_2(x^k)$
	with some $(v^k_1, v^k_2)\in\widehat{\partial}\varphi_1(x^k)\times \widehat{\partial}\varphi_2(x^k)$. Since $x^k$ belongs to $\mathcal{B}(x)$, where $\widehat{\partial}\varphi_2$ is uniformly bounded, the boundedness of $\{ v_2^k: k\in\mathbb{N}\}$ yields some subsequence $v^{k_j}_2 \in \widehat{\partial}\varphi_2(x^{k_j})$ converging to some $v_2\in\partial\varphi_2(x)$. By passing to the limit with $j\to\infty$, we then obtain $v_1:=\lim_{j\to\infty} v^{k_j}_1 = \lim_{j\to\infty} \varphi_2(x^{k_j})w^{k_j}+\tau(x^{k_j})v^{k_j}_2 = a_2w+\frac{a_1}{a_2}v_2$. Besides, we have  $v_1\in\partial\varphi_1(x)$ due to $v^k_1\in\widehat{\partial}\varphi_1(x^k)$, $v^k_1\to v_1$, $x^k\to x$ and the continuity of $\varphi_1$ on $\dom(\varphi_1)$. Consequently, for any $w\in\partial\tau(x)$, there exists some $(v_1,v_2)\in \partial\varphi_1(x)\times \partial\varphi_2(x)$ such that $w = \frac{1}{a_2}(v_1-\frac{a_1}{a_2}v_2)$, and hence the relation \eqref{eq: subdiff of tau} holds. Specially when $\varphi_2$ is continuously differentiable around $x$, the relation \eqref{eq: subdiff of tau} holds with equality since \eqref{eq: z0901 2330} becomes equality. Item \StatementNum{2} is derived from \Cref{ppsition:2.2} \StatementNum{2}, while the proof is omitted here for the similarity to the proof of Item \StatementNum{1}.
\end{proof}

To end this subsection, we recall some useful properties on the subdifferential of a convex real-valued function. Let $\varphi:\mathbb{R}^n\to\mathbb{R}$ be convex. Then $\varphi$ is locally Lipschitz continuous \cite[Corollary 8.31, Theorem 8.29]{Bauschke-Combettes:11}, and $\cup_{x\in\mathcal{A}} \partial \varphi(x)$ is nonempty and bounded on any compact set $\mathcal{A}\subseteq \mathbb{R}^n$ (\cite[Proposition 5.4.2]{Bertsekas:Convexoptimizationtheory}). For a proper closed convex function $\varphi:\mathbb{R}^n\to(-\infty,+\infty]$, the conjugate $\varphi^*$ is also a proper closed convex function (\cite[Theorem 11.1]{Rockafellar2004Variational}) and $(\varphi^*)^* = \varphi$ (\cite[Theorem 4.8]{Beck-Amir:2017SIAM}). Moreover, it is known that the following equivalence holds (\cite[Proposition 11.3]{Rockafellar2004Variational}): $$\innerP{x}{y} = \varphi(x)+\varphi^*(y) \Leftrightarrow y\in\partial \varphi(x) \Leftrightarrow x\in\partial \varphi^*(y).$$


\subsection{KL property}

We now recall KL property, which has been used extensively in the convergence analysis of various first-order methods.


\begin{definition}[KL property and KL exponent \cite{Attouch-bolt-redont-soubeyran:2010}]\label{Def:KL_property}
	A proper function $\varphi:\mathbb{R}^n\to(-\infty,+\infty]$ is said to satisfy the KL property at $x\in\mathrm{dom}(\partial \varphi)$ if there exist $\epsilon\in(0,+\infty]$, $\delta>0$ and a continuous concave function $\phi:[0,\epsilon) \to \mathbb{R}_+:=[0,+\infty)$ such that:
	\begin{enumerate}[label = {\upshape(\roman*)}]
		\item $\phi(0)=0$;
		\item $\phi$ is continuously differentiable on $(0,\epsilon)$ with $\phi'>0$;
		\item For any $z\in \mathcal{B}_{\varphi}^{\epsilon}(x,\delta)$, there holds $\phi'(\varphi(z)-\varphi(x))\mathrm{~dist}(0,\partial\varphi(z)) \geq 1$.
	\end{enumerate}
	If $\varphi$ satisfies the KL property at $x\in\dom(\partial \varphi)$ and the $\phi$ can be chosen as $\phi(z) = a_0 z^{1-\theta}$ for some $a_0>0$ and $\theta\in [0,1)$, then we say that $\varphi$ satisfies the KL property at $x$ with the exponent $\theta$.
\end{definition}

A proper function $\varphi:\mathbb{R}^n\to (-\infty,+\infty]$ is called a KL function if it satisfies the KL property at any point in $\dom(\partial \varphi)$, and a proper function $\varphi$ satisfying the KL property with exponent $\theta\in [0,1)$ at every point in $\dom(\partial \varphi)$ is called a KL function with exponent $\theta$. For connections between the KL property and the well-known error bound theory \cite{Luo-Pang:1994Mathematical_Programming, Pang:1997Mathematical_Programming}, we refer the interested refers to \cite{Bolte-Nguyen-Peypouquet:2017Mathematical_Programming, Li-Pong:2018Foundations_of_computational_mathematics}.

A wide range of functions are KL functions. Among those functions, the proper lower semicontinuous semialgebraic functions (see \cite[Theorem 3]{Bolte-Sabach-Teboulle:MP:2014}) cover most frequently appearing functions in applications.
A function $\varphi:\mathbb{R}^n\to (-\infty,+\infty]$ is said to be semialgebraic if its graph Graph$(\varphi):=\{(x,s)\in\mathbb{R}^n\times\mathbb{R}:s=\varphi(x)\}$ is a semialgebraic subset of $\mathbb{R}^{n+1}$, that is, there exist a finite number of real polynomial functions $G_{ij}$, $H_{ij}:\mathbb{R}^{n+1}\to\mathbb{R}$ such that
\begin{equation*}
	\Graph(\varphi)=\bigcup^p_{j=1}\,\bigcap^q_{i=1}\,\{z\in\mathbb{R}^{n+1}:G_{ij}(z)=0,~H_{ij}(z)<0 \}.
\end{equation*}
In addition, \cite{Attouch-bolt-redont-soubeyran:2010} and \cite[Theorem 3.1]{Bolte-Daniilidis-Lewis:2007SIAMOPT} pointed out that a proper closed semialgebraic function is a KL function with some exponent $\theta\in [0,1)$.

The following lemma regards the result of the uniformized KL property.



\begin{lemma}[Uniformized KL property]\label{lemma: Uniformized KL property}
	Let $\Upsilon\subseteq\mathbb{R}^n$ be a compact set, and the proper function $\varphi:\mathbb{R}^n\to(-\infty,+\infty]$ be constant on $\Upsilon$. 
	\begin{enumerate}[label = {\upshape(\roman*)}]
		\item If $\varphi$ satisfies the KL property at each point of $\Upsilon$, then there exist $\epsilon>0$ and a continuous concave function $\phi:[0,\epsilon)\to[0,+\infty)$ satisfying Definition \ref{Def:KL_property} \StatementNum{1} and \StatementNum{2}; besides, there exists $\delta>0$ such that
		$\phi'(\varphi(z)-\varphi(x))\mathrm{~dist}(0,\partial\varphi(z)) \geq 1$
		holds for any $x\in\Upsilon$ and $z\in\mathcal{B}_{\varphi}^{\epsilon}(x,\delta)$.
		\item If $\varphi$ satisfies the KL property at each point of $\Upsilon$ with the exponent $\theta\in[0,1)$, then there exist $\epsilon,\delta,c>0$ such that
		$\mathrm{~dist}(0,\partial\varphi(z)) \geq c (\varphi(z)-\varphi(x))^{\theta}$
		holds for any $x\in\Upsilon$ and $z\in\mathcal{B}_{\varphi}^{\epsilon}(x,\delta)$.
	\end{enumerate}
\end{lemma}

\Cref{lemma: Uniformized KL property} can be found in \cite[Lemma 6]{Bolte-Sabach-Teboulle:MP:2014} and \cite[Lemma 2.2]{Yu-Li-Pong:2021FOCM} with an additional assumption that $\varphi$ is lower semicontinuous on $\mathbb{R}^n$. However, we notice that this assumption is not used in the proof of \cite[Lemma 6]{Bolte-Sabach-Teboulle:MP:2014} and \cite[Lemma 2.2]{Yu-Li-Pong:2021FOCM}. Thus, we present \Cref{lemma: Uniformized KL property} without assuming the lower semicontinuity of $\varphi$. Thanks to \Cref{lemma: Uniformized KL property}, we generalize the framework proposed in \cite[Theorem 1]{Banert-Bot:2019MP} for proving global sequential convergence in the next proposition.

\begin{proposition}\label{proposition: KL convergence framework}
	Let $H:\mathbb{R}^n\times\mathbb{R}^m\to (-\infty,+\infty]$ be proper, and $\{a_k:k\in\mathbb{N} \}$ be a nonnegative scalar sequence. Consider a bounded sequence $\{(x^k,y^k):k\in\mathbb{N} \}$ satisfying the following three conditions:
	\begin{enumerate}[label = {\upshape(\roman*)}]
		\item (Sufficient decrease condition.) There exist $C_1>0$ and $K_1>0$ such that
		\begin{equation*}
			H(x^{k+1},y^{k+1}) + C_1\left( \|x^{k+1}-x^k\|_2^2 + a_{k}\|y^{k+1}-y^k\|^2_2 \right) \leq H(x^k,y^k)
		\end{equation*}
		holds for any $k\geq K_1$;
		\item (Relative error condition.) There exist $C_2>0$ and $K_2>0$ such that
		\begin{equation*}
			\|w^{k+1}\|_2 \leq C_2\left( \|x^{k+1}-x^k\|_2 + \sqrt{a_{k}}\|y^{k+1}-y^k\|_2 \right)
		\end{equation*}
		holds with some $w^{k+1} \in \partial H(x^{k+1},y^{k+1})$ for any $k\geq K_2$;
		\item (Continuity condition.) $\lim_{k\to\infty} H(x^k,y^k) = \xi $ exists, and $H\equiv \xi \text{  holds on  }\Upsilon,$ where $\Upsilon$ is the set of accumulation points of $\{(x^k,y^k):k\in\mathbb{N} \}$.
	\end{enumerate}	
	If $H$ satisfies the KL property at each point of $\Upsilon$, 
	then there hold:
	\begin{enumerate}[label = {\upshape(\roman*)}]
		\item $\sum_{k=0}^{\infty}{ (\|x^{k+1}-x^k\|_2 + \sqrt{a_{k}}\|y^{k+1}-y^k\|_2)  }<+\infty$;
		\item  $\lim_{k\to\infty} x^k = x^{\star}$ for some $x^{\star}\in\mathbb{R}^n$;
		\item $0\in\partial H(x^{\star},y^{\star})$ for any $(x^{\star},y^{\star})\in\Upsilon$.
	\end{enumerate}
\end{proposition}
We prove this Proposition \ref{proposition: KL convergence framework} in Appendix \ref{Appendix: the proof of KL convergence}.

\section{The relationship between problems \eqref{problem:root} and \eqref{problem:root primal-dual}}\label{section: Connections between primal and primal-dual problems}

In this section, we establish the relationship between optimal solutions and critical points, as well as KL exponents for the extended objectives of problem \eqref{problem:root} and problem \eqref{problem:root primal-dual}. The extended objective $F:\mathbb{R}^n\to [0,+\infty]$ of problem \eqref{problem:root} is defined at $x\in\mathbb{R}^n$ as
\begin{equation}
	\label{definition: function F}
	F(x) := \begin{cases}
		\frac{f(x)+h(x)}{g(x)}, & \text{if } x\in\Omega\cap\dom(f),\\
		+\infty, & \text{else.}
	\end{cases}
\end{equation}
By introducing $\zeta(x) := f(x)+h(x)$ and $\eta(x,y) := \innerP{x}{y}-g^*(y)$, 
the extended objective  $Q:\mathbb{R}^n\times \mathbb{R}^n\to [0,+\infty]$ of problem \eqref{problem:root primal-dual} is defined at $(x,y)\in \mathbb{R}^n\times \mathbb{R}^n$ as
\begin{equation}
	\label{definition: function Q}
	Q(x,y) := \begin{cases}
		\frac{\zeta(x)}{\eta(x,y)}, & \text{if } (x,y)\in\dom(f)\times\dom(g^*) \text{ and } \eta(x,y)>0,\\
		+\infty, & \text{else.}
	\end{cases}
\end{equation}
It is worth noting that: \StatementNum{1} For any $x\in\dom(F)$, we have $\{x\} \times\partial g(x) \subseteq \dom(Q) $; \StatementNum{2} For any $(x,y)\in\dom(Q)$, there holds $x\in\dom(F)$. Thanks to the above extended objectives, problems \eqref{problem:root} and \eqref{problem:root primal-dual} can be rewritten as $\min\{F(x):x\in\mathbb{R}^n \}$ and $\min\{Q(x,y):(x,y)\in\mathbb{R}^n\times\mathbb{R}^n \}$, respectively.
We next prove the equivalence of problems \eqref{problem:root} and \eqref{problem:root primal-dual} in the following proposition.
\begin{proposition}
	The functions $F$ and $Q$ respectively given by \eqref{definition: function F} and \eqref{definition: function Q} have the same infimum. Let $(x^{\star},y^{\star})\in\dom(Q)$. Then, $x^{\star}$ is a global minimizer of $F$ with $F(x^{\star})y^{\star}\in F(x^{\star}) \partial g(x^{\star})$, if and only if $(x^{\star},y^{\star})$ is a global minimizer of $Q$.
\end{proposition}

\begin{proof}
	Since $f+h\geq 0$ and $\partial g(x)\neq \emptyset$ holds for each $x\in\mathbb{R}^n$, the \textit{Fenchel-Young Inequality} implies
	\begin{equation}\label{eq: z1110 1621}
		\begin{split}
			&\inf \,\{ F(x):x\in\dom(F) \}\\
			& \indent =\inf\left\{ \frac{f(x)+h(x)}{\innerP{x}{y}-g^*(y)}: x\in\dom(f), \innerP{x}{y}-g^*(y)>0 , y\in\partial g(x)  \right\}\\
			& \indent= \inf\,\{ Q(x,y):(x,y)\in\dom(Q) \}.
		\end{split}
	\end{equation}
	
	Let $x^{\star}$ be a global minimizer of $F$ and $F(x^{\star})y^{\star}\in F(x^{\star}) \partial g(x^{\star})$. Then, when $F(x^{\star})>0$, there holds
	\begin{equation}\label{eq: z1110 1622}
		\frac{f(x^{\star})+h(x^{\star})}{\innerP{x^{\star}}{y^{\star}}-g^*(y^{\star})} = \frac{f(x^{\star})+h(x^{\star})}{g(x^{\star})} = \inf\{ F(x):x\in\dom(F) \},
	\end{equation}
	hence it follows from \eqref{eq: z1110 1621} and \eqref{eq: z1110 1622} that $(x^{\star},y^{\star})$ is a global minimizer of $Q$. Besides, $(x^{\star},y^{\star})$, taking $Q(x^{\star},y^{\star})=0$, is still a global minimizer of $Q$ when $F(x^{\star})=0$.  
	Conversely, suppose $(x^{\star},y^{\star})\in\dom(Q)$ is a global minimizer of $Q$. Then we have
	\begin{equation}\label{eq: z1110 1623}
		\frac{f(x^{\star})+h(x^{\star})}{g(x^{\star})} \leq \frac{f(x^{\star})+h(x^{\star})}{\innerP{x^{\star}}{y^{\star}}-g^*(y^{\star})} = \inf\{ Q(x,y):(x,y)\in\dom(Q) \}.
	\end{equation}
	By combining \eqref{eq: z1110 1621} and \eqref{eq: z1110 1623}, we deduce that $x^{\star}$ is a global minimizer of $F$. Therefore, it is derived from the first relation of \eqref{eq: z1110 1623}, which holds with equality actually, that
	$	F(x^{\star})(\innerP{x^{\star}}{y^{\star}} - g^*(y^{\star}))
	= f(x^{\star})+h(x^{\star}) = F(x^{\star})g(x^{\star}),$
	then $F(x^{\star})y^{\star}\in F(x^{\star}) \partial g(x^{\star})$ follows. This completes the proof.
\end{proof}

We next shall show the connections between critical points of $F$ and stationary points of $Q$. To this end, we recall the definition of critical points of problem \eqref{problem:root}.

\begin{definition}[Critical points of $F$ {\cite[Definition 3.4]{NaZhang-QiaLi:2022SIAM-OPT}}]\label{Def: critical points}
	Let $F$ be defined as \eqref{definition: function F}. Then $x^{\star}\in\dom(F)$ is said to be a critical point of $F$ if
	$$0 \in \widehat{\partial} f(x^{\star}) + \nabla h(x^{\star}) - F(x^{\star}) \partial g(x^{\star}).$$
\end{definition}

It is demonstrated by \cite[below Definition 3.4]{NaZhang-QiaLi:2022SIAM-OPT} that the statement that $x^{\star}\in \dom(F)$ is a critical point of $F$ coincides with the statements that $x^{\star}\in \dom(F)$ is a stationary point of $F$, i.e., $0\in \widehat{\partial} F(x^{\star})$, when the denominator $g$ is differentiable. The following proposition shows the equivalence between critical points of $F$ and stationary points of $Q$, when $g^*$ is continuous on its domain.


\begin{proposition}\label{proposition: the equivalence of F and Q critical point}
	Let $(x^{\star},y^{\star})\in\dom(Q)$, and $g^*$ be continuous on $\dom(g^*)$. Then, $x^{\star}$ is a critical point of $F$ with $F(x^{\star})y^{\star}\in F(x^{\star})\partial g(x^{\star})$, if and only if $(x^{\star},y^{\star})$ is a stationary point of $Q$, i.e., $0\in\widehat{\partial} Q(x^{\star},y^{\star})$.	
\end{proposition}

\begin{proof}
	It is trivial for \Cref{proposition: the equivalence of F and Q critical point} when $F(x^{\star}) =0$, since $x^{\star}$ and $(x^{\star},y^{\star}) $ are global minimizers of $F$ and $Q$, respectively.
	Now we deal with the case when $F(x^{\star}) >0$. To this end, we introduce an auxiliary function for $(x^{\star},y^{\star})\in\dom(Q)$ by
	\begin{equation*}
		\phi_{(x^{\star},y^{\star})}(x,y) := \eta(x^{\star},y^{\star}) \zeta(x)-\zeta(x^{\star})\eta(x,y).
	\end{equation*}
	Then we have
	\begin{align}
		\liminf\limits_{\substack{(x,y)\in\dom(Q)\\(x,y)\to (x^{\star},y^{\star})}} \frac{Q(x,y)-Q(x^{\star},y^{\star})}{\|(x,y)-(x^{\star},y^{\star})\|_2}
		=
		\liminf\limits_{\substack{(x,y)\in\dom(Q)\\(x,y)\to (x^{\star},y^{\star})}} \frac{\phi_{(x^{\star},y^{\star})}(x,y)}{\eta(x^{\star},y^{\star})\eta(x,y)\|(x,y)-(x^{\star},y^{\star})\|_2}\notag \\
		=
		\eta^{-2}(x^{\star},y^{\star})
		\liminf\limits_{\substack{(x,y)\in\dom(\phi_{(x^{\star},y^{\star})})\\(x,y)\to (x^{\star},y^{\star})}} \frac{\phi_{(x^{\star},y^{\star})}(x,y)}{\|(x,y)-(x^{\star},y^{\star})\|_2}, \label{eq: z1112 5}
	\end{align}
	where the last equality follows from the continuity of $g^*$ on $\dom(g^*)$ and $\eta(x^{\star},y^{\star})>0$.
	Besides, for any $(x,y)\in\dom(\phi_{(x^{\star},y^{\star})}) = \dom(f)\times \dom(g^*)$, there holds
	\begin{equation*}
		\widehat{\partial} \phi_{(x^{\star},y^{\star})}(x,y) =
		\Big (  \eta(x^{\star},y^{\star})\, \widehat{\partial}\zeta(x) - \zeta(x^{\star})y \Big )
		\times
		\Big ( \widehat{\partial} (\zeta(x^{\star})g^*)\,(y) -\zeta(x^{\star})x \Big ).
	\end{equation*}
	
	If $x^{\star}$ is a critical point of $F$ with $y^{\star}\in\partial g(x^{\star})$, there holds $0\in\widehat{\partial}\phi_{(x^{\star},y^{\star})}(x^{\star},y^{\star})$, and thus we derive $0\in\widehat{\partial}Q(x^{\star},y^{\star})$ from \eqref{eq: z1112 5} with the fact that $\phi_{(x^{\star},y^{\star})}(x^{\star},y^{\star})=0$. Conversely, based on \eqref{eq: z1112 5}, the statement $0\in\widehat{\partial}Q(x^{\star},y^{\star})$ indicates $0\in\widehat{\partial}\phi_{(x^{\star},y^{\star})}(x^{\star},y^{\star})$, with which we deduce that $x^{\star}$ is a critical point of $F$ and $y^{\star}\in\partial g(x^{\star})$ owing to $\zeta(x^{\star}) = F(x^{\star})g(x^{\star}) >0$. This completes the proof.
\end{proof}

In particular when $g$ is continuously differentiable on $\Omega$ and $g^*$ satisfies the calmness condition on $\dom(g^*)$, we have $\eta(x,\nabla g(x)) = g(x)$ and $x\in\partial g^*(\nabla g(x))$ for any $x\in\Omega$. Then, there holds for any $x\in\dom(F)$ that
\begin{equation}\label{eq: z03051644}
	\dist(0,\partial F(x)) = \dist(0, \partial Q(x,\nabla g(x))).
\end{equation}
Specifically speaking, in the case when $\zeta(x)>0$, the equation \eqref{eq: z03051644} holds due to $\partial F(x) = \frac{1}{g(x)}(\partial \zeta(x)-F(x)\nabla g(x)) $ and $\partial Q(x,\nabla g(x))
= \partial F(x) \times \frac{F(x)}{g(x)} (\partial g^*(\nabla g(x))-x )$, which is indicated by \Cref{Corollary: limiting subdiff of frac}. In the case when $\zeta(x)=0$, since $x$ and $(x,\nabla g(x)) $ become global minimizers of $F$ and $Q$, respectively, we have $0\in\partial F(x)$ and $0\in\partial Q(x,\nabla g(x))$. Thus, it follows that $\dist(0,\partial F(x)) = \dist(0, \partial Q(x,\nabla g(x)))=0$. Owing to the closedness of the limiting subdifferential, we further deduce from \eqref{eq: z03051644} that $0\in\partial F(x)$ coincides with $0\in\partial Q(x,\nabla g(x))$.

The next theorem demonstrates the connections between KL exponents of $F$ and $Q$.
\begin{theorem}\label{theorem: KL exponents Q to F}
	Let $\overline{x}\in \dom(\partial F)$. Suppose that $g^*$ satisfies the calmness condition on $\dom(g^*)$. If $Q$ satisfies the KL property with the exponent $\theta\in[0,1)$ at $(\overline{x},\overline{y})$ for any $\overline{y}\in\partial g(\overline{x})$, then $F$ satisfies the KL property with the same exponent $\theta$ at $\overline{x}$.
\end{theorem}

\begin{proof}
	We first claim that 
	\begin{equation}\label{eq: z03051721}
		\lim\limits_{x\to\overline{x}} \sup\Big\{ \dist(y,\partial g(\overline{x})): y\in\partial g(x) \Big\} = 0.
	\end{equation}
	We prove the equation \eqref{eq: z03051721} by contradiction. Suppose \eqref{eq: z03051721} does not hold, then there exist some $\epsilon>0$, $x^k\to\overline{x}$ and $y^k\in\partial g(x^k)$ such that $\dist(y^k,\partial g(\overline{x}))>\epsilon$, since $\partial g(x)$ is nonempty and compact for all $x\in\mathbb{R}^n$. The boundedness of $ \bigcup_{k\in\mathbb{N}} \partial g(x^k)$ implies some subsequence $y^{k_j}\to y^{\star}$, and $\dist(y^{\star},\partial g(\overline{x}))\geq \epsilon$ follows. However, by the definition of limiting subdifferentials, we have $y^{\star}\in\partial g(\overline{x})$ upon the fact that $y^{k_j}\in\partial g(x^{k_j})$ and $x^{k_j}\to\overline{x}$. Thus, the contradiction implies that \eqref{eq: z03051721} holds.
	
	Now we go back to the proof of the KL exponents of $F$. Since $Q\equiv F(\overline{x})$ holds on the compact set $\{\overline{x} \}\times\partial g(\overline{x})$, the function $Q$ has the uniformized KL exponent with the help of \Cref{lemma: Uniformized KL property} \StatementNum{2}. That is, there exists $c,\delta_1>0$ and $\epsilon\in (0,+\infty]$ such that
	\begin{equation}\label{eq: z0903 2320}
		\dist(0,\partial Q(x,y))\geq c(Q(x,y)-F(\overline{x}))^{\theta}
	\end{equation}
	holds for any $(x,y)\in \mathcal{U}:= \{(x,y):\dist((x,y),(\overline{x},\partial g(\overline{x})) ) \leq \delta_1  \}$ satisfying $Q(\overline{x},\overline{y}) < Q(x,y) < Q(\overline{x},\overline{y})+\epsilon $. With regard to $g(\overline{x}) >0$, we assume further that the $\delta_1>0$ is small enough so that $g>0$ holds on $\mathcal{U}$. Besides, the equation \eqref{eq: z03051721} implies some $\delta_2>0$ such that $\sup\{ \dist(y,\partial g(\overline{x})): y\in\partial g(x) \}\leq \delta_1$ holds for any $x\in \mathcal{B}(\overline{x},\delta_2)$. By taking $\delta:= \frac{1}{\sqrt{2}}\min\{\delta_1,\delta_2 \}$. one can verify that \eqref{eq: z0903 2320} holds for any $x\in \mathcal{B}_{F}^\epsilon(\overline{x},\delta)$ and $y\in\partial g(x)$, meanwhile $\zeta>0$ and $g>0$ holds on $\mathcal{B}_{F}^\epsilon(\overline{x},\delta)$. Owing to \Cref{Corollary: limiting subdiff of frac} \StatementNum{1}, it holds for any $x\in \mathcal{B}_{F}^\epsilon(\overline{x},\widetilde{\delta})\cap \dom(\partial F)$ that
	\begin{align*}
		&\dist(0,\partial F(x) ) 
		\geq \dist\left(0,\frac{1}{g(x)}(\partial\zeta(x)-F(x)\partial g(x))\right)\\
		&\overset{(\uppercase\expandafter{\romannumeral 1})}{=}
		\min\left\{\dist\left(0,\frac{1}{g(x)}(\partial\zeta(x)-F(x)y)\right):y\in\partial g(x)  \right\}  \\
		&= \min\left\{\dist\left(0,\frac{1}{\eta(x,y)}(\partial\zeta(x)-Q(x,y)y,Q(x,y)(\partial g^*(y)-x))\right):y\in\partial g(x)\right\} \\
		&\overset{(\uppercase\expandafter{\romannumeral 2})}{=} \min\{\dist(0,\partial Q(x,y)):y\in\partial g(x) \}\\
		&\overset{(\uppercase\expandafter{\romannumeral 3})}{\geq}
		\min\{c(Q(x,y)-F(\overline{x}))^{\theta}: y\in\partial g(x)  \} 
		=c(F(x)-F(\overline{x}))^{\theta},
	\end{align*}
	where $(\uppercase\expandafter{\romannumeral 1})$ follows from the compactness of $\partial g(x)$, and $(\uppercase\expandafter{\romannumeral 2})$ comes from \Cref{Corollary: limiting subdiff of frac} \StatementNum{2}, and $(\uppercase\expandafter{\romannumeral 3})$ holds thanks to \eqref{eq: z0903 2320}. This completes the proof.
\end{proof}


\section{A framework of multi-proximity gradient algorithms} \label{section: MPGA}

In this section, we propose a framework of multi-proximity gradient algorithms for solving problem \eqref{problem:root} and show that it has several desired properties. We first introduce the notion of proximity operators. For a proper closed function $\varphi:\mathbb{R}^n\to(-\infty,+\infty]$, the proximity operator of $\varphi$ at $x\in\mathbb{R}^n$, denoted by $\prox_{\varphi}(x)$, is defined by
\begin{equation*}
	\prox_{\varphi}(x) := \arg\min\left\{\varphi(z)+\frac{1}{2}\|z-x\|^2_2: z\in\mathbb{R}^n \right\}.
\end{equation*}
The proximity operator $\prox_{\varphi}(x)$ is single-valued if $\varphi$ is convex and may be set-valued when $\varphi$ is nonconvex. Below we present a sufficient condition for a stationary point of $Q$ and a critical point of $F$ using proximity operators of $f_i$, $i\in\mathbb{N}_N$ and $g^*$.

\begin{proposition}\label{prop: stationary point of prox}
	If $(x^{\star},y^{\star})\in\Omega\times\mathbb{R}^n$ satisfies	
	\begin{align}
		&y^{\star}=\prox_{\alpha_0 g^*}(y^{\star}+\alpha_0 x^{\star}),\label{eq: z1127 2200}\\
		&x^{\star}_i\in\prox_{\alpha_i f_i}(x^{\star}_i-\alpha_i\nabla_i h(x^{\star})+\alpha_i Q(x^{\star},y^{\star})y^{\star}_i),~~i\in\mathbb{N}_N,\label{eq: z1127 2201}
	\end{align}
	for some $\alpha_i>0$, $i\in\mathbb{N}_N^0$, then $x^{\star}$ is a critical point of $F$ with $y^{\star}\in\partial g(x^{\star})$, and hence $(x^{\star},y^{\star})\in\dom(Q)$ is a stationary point of $Q$ if $g^*$ is continuous on $\dom(g^*)$.
\end{proposition}

\begin{proof}
	By the definition of proximity operators and the generalized \textit{Fermat's Rule}, \eqref{eq: z1127 2200} implies $0\in\partial g^*(y^{\star})-x^{\star}$. Therefore, we obtain $y^{\star}\in\partial g(x^{\star})$, $\eta(x^{\star},y^{\star})=g(x^{\star})>0$ and $Q(x^{\star},y^{\star})=F(x^{\star})$. Moreover, \eqref{eq: z1127 2201} implies
	\begin{equation}\label{eq: z1127 2209}
		0\in\widehat{\partial} f_i(x^{\star}_i) + \nabla_i h(x^{\star}) - F(x^{\star}) y^{\star}_i,~~~i\in\mathbb{N}_N.
	\end{equation}
	In view of \Cref{Def: critical points} and $\widehat{\partial}f(x^{\star}) = \widehat{\partial}f_1(x_1^{\star})\times \widehat{\partial}f_2(x_2^{\star})\times ... \times \widehat{\partial}f_N(x_N^{\star})$, we deduce that $x^{\star}$ is a critical point of $F$ with $y^{\star}\in\partial g(x^{\star})$. Besides, with the help of \Cref{proposition: the equivalence of F and Q critical point}, we obtain $0\in\widehat{\partial} Q(x^{\star},y^{\star})$. This completes the proof.
\end{proof}

Motivated by \Cref{prop: stationary point of prox}, we propose a framework of multi-proximity gradient algorithms (MPGA) for solving problem \eqref{problem:root}, which is described in \Cref{alg:MPGA}. Each iteration of MPGA starts with picking an index $i\in\mathbb{N}_N^0$. If the selected $i\in\mathbb{N}_N$, then the algorithm solves a proximal subproblem associated with $x_i$ while keeping the other blocks at their last updated values. Also, we take advantage of the nonmonotone line search scheme in \cite{Chen-Lu-Pong:2016SIAM-OPT, Lu-Xiao:SIAM2017, Tono-Takeda-Gotoh:2017DCforSparseOpt, Wright-Nowak-Figueiredo:IEEE-TSP2009} to find an appropriate step-size which can ensure a certain progress is achieved at the iteration. If the selected $i=0$, then the MPGA simply performs a proximity step of $g^*$ with respect to $y$, while fixing $x$ at its newest updated value. This framework of MPGA is very flexible since the way to choose an index $i\in\mathbb{N}_N^0$ is not specified. We next discuss the
well-definitedness of MPGA. Specifically, according to MPGA, we shall show that the nonmonotone line search can terminate in finite steps, and the sequence $(x^{(t)},y^{(t)})$ generated by MPGA falls into $\dom(Q)$, in the sense that the denominator $\eta$ which the objective $Q$ involves is always positive at each iteration $(x^{(t)},y^{(t)})$. To this end, we introduce the following assumption.

\begin{algorithm}
	\caption{A framework of multi-proximity gradient algorithm (MPGA) for solving problem \eqref{problem:root}.}
	\label{alg:MPGA}
	\renewcommand{\arraystretch}{1.2}
	\begin{tabular}{ll}
		\textbf{Step 0.} & Input $x^{(0)}\in\dom(F)$, $y^{(0)}\in\partial g(x^{(0)})$, \\
		& { $0<\underline{\alpha} \leq  \overline{\alpha}$, $\sigma>0$, $0<\gamma<1$ and an integer $M\geq 0$.  Set $t \leftarrow 0$.} \\[5pt]
		\textbf{Step 1.} & Compute $Q_{(t)} = Q(x^{(t)},y^{(t)})$.\\
		& Compute $l(t) = \max\{j\leq t:Q_{(j)}=\max\{Q_{(s)}:[t-M]_+\leq s\leq t \}  \}$.\\[5pt] 
	\textbf{Step 2.} & Pick $i\in \mathbb{N}^0_N$. If $i=0$, set $\alpha := \widetilde{\alpha}_Y \in [\underline{\alpha},\overline{\alpha}]$ and	go to Step 2-Y; \\
	& otherwise, set $\alpha := \widetilde{\alpha}_{(t)} \in [\underline{\alpha},\overline{\alpha}]$ and go to Step 2-X. \\
	{~~~Step 2-Y.} & Compute $y^{(t)}(\alpha) = \prox_{\alpha g^*}(y^{(t)} + \alpha x^{(t)})$. \\
	& Set $(x^{(t+1)},y^{(t+1)})\leftarrow (x^{(t)},y^{(t)}(\alpha)) $. Go to \textbf{Step 3}. \\
	{~~~Step 2-X.} & Compute $x^{(t)}_i(\alpha) \in \prox_{\alpha f_i}\left(x^{(t)}_i-\alpha\nabla_i h(x^{(t)})+\alpha Q_{(t)} y^{(t)}_i\right)$.\\
	& Let $x^{(t),i}(\alpha)=\left[x^{(t)}_1;x^{(t)}_2;...;x^{(t)}_{i-1}; 
	~ x^{(t)}_i(\alpha);~
	x^{(t)}_{i+1};...;x^{(t)}_N \right]  $. \\
	& If $\zeta(x^{(t),i}(\alpha)) + \frac{\sigma}{2}\|x^{(t),i}(\alpha) - x^{(t)}\|^2_2 \leq Q_{(l(t))} \eta(x^{(t),i}(\alpha),y^{(t)})$. \\
	&        ~~ then set $(x^{(t+1)},y^{(t+1)})\leftarrow (x^{(t),i}(\alpha),y^{(t)}) $ and go to \textbf{Step 3}; \\
	& Else, set $\alpha\leftarrow \alpha\gamma$ and go to Step 2-X. \\[5pt]
	\textbf{Step 3.} & Record $i_{(t)} := i$ and $\alpha_{(t)} := \alpha$, set $t\leftarrow t+1$, and go to \textbf{Step 1}.      
\end{tabular}%
\end{algorithm}

\begin{assumption}\label{Assumption: X is compact}
The level set $\mathcal{X} := \{x\in\mathbb{R}^n: F(x)\leq F(x^{(0)}) \}$ is compact.
\end{assumption}

We remark that boundedness of the level set associated with the extended objective is a standard assumption in the literatures of nonconvex optimization, while its closedness automatically holds in many cases due to the lower semi-continuity of the extended objective. However, $F$ in \eqref{definition: function F} may be not lower semicontinuous in general, which necessitates the closedness condition in \Cref{Assumption: X is compact}. In \cite[Proposition 4.4]{NaZhang-QiaLi:2022SIAM-OPT}, it is shown that $F$ is lower semicontinuous once $\zeta$ and $g$ do not attain zero simultaneously. This condition, in particular, is satisfied for the $L_1/S_K$ sparse signal recovery model \eqref{problem:L1dSK}, where the numerator always takes positive values. Combining this observation and invoking the constraint $\{x\in\mathbb{R}^n:\underline{x}\leq x\leq \overline{x} \} $, we conclude that \Cref{Assumption: X is compact} is fulfilled by MPGA for model \eqref{problem:L1dSK} with an arbitrary initial point $x^{(0)}$. On the other hand, if $\zeta$ and $g$ do attain zero simultaneously, we can still ensure the closedness of $\mathcal{X}$ by properly choosing $x^{(0)}$ as shown in the following proposition.

\begin{proposition}\label{proposition: X is closed}
Suppose that the set $\mathcal{O}:= \{x\in\mathbb{R}^n:\zeta(x)=g(x)=0 \}$ is nonempty. Then, $\mathcal{X}$ is closed, if $x^{(0)}\in\mathbb{R}^n$ satisfies
\begin{equation}\label{eq:assumption for F0}
F(x^{(0)}) < \inf\left\{ \Liminf_{z\to x} F(z):x\in\mathcal{O} \right\}.
\end{equation}
\end{proposition}

\begin{proof}
We shall show that $\mathcal{X}$ is closed by verifying that each accumulation point of $\mathcal{X}$ belongs to $\mathcal{X}$ when \eqref{eq:assumption for F0} is satisfied. Let $x^{\star}$ be an accumulation point of $\mathcal{X}$ and some sequence $\{x^k:k\in\mathbb{N} \}\subseteq \mathcal{X} $ converge to $x^{\star}$. Then there holds $\zeta(x^k)\leq F(x^{(0)})g(x^k)$ for any $k\in\mathbb{N}$ due to $x^k\in\mathcal{X}$, and $\zeta(x^{\star})\leq F(x^{(0)})g(x^{\star})$ follows from the lower semi-continuity of $\zeta$ and the continuity of $g$.
Assume that $g(x^{\star}) = 0$. This together with $\zeta(x^{\star})\leq F(x^{(0)})g(x^{\star})$ leads to $x^{\star}\in\mathcal{O}$. By invoking this and \eqref{eq:assumption for F0}, we get $F(x^{(0)}) < \Liminf_{k\to\infty} F(x^k)$, which contradicts the fact that $\Liminf_{k\to\infty} F(x^k)\leq F(x^{(0)})$. Hence, we conclude that $g(x^{\star}) > 0$ and thus $x^{\star}\in \mathcal{X}$.
\end{proof}

For the $L_1/L_2$ sparse signal recovery model \eqref{problem:L1dL2}, the zero vector is the unique point at which both $\zeta$ and $g$ vanish. According to \Cref{proposition: X is closed} and model \eqref{problem:L1dL2}, if $x^{(0)}\in\mathbb{R}^n$ satisfies $F(x^{(0)}) <  \Liminf_{z\to 0} F(z)  = 1+\frac{\lambda}{2}\|b\|_2^2$, then the level set $\mathcal{X}$ is closed. This with the boundedness of the constrain $\{x\in\mathbb{R}^n:\underline{x}\leq x\leq \overline{x} \} $ tells that \Cref{Assumption: X is compact} is satisfied by MPGA for model \eqref{problem:L1dL2} given a suitable initial point $x^{(0)}$.

\begin{remark}\label{Remark: h is global Lips and sup of g}
\Cref{Assumption: X is compact} directly yields the following two results, which will be frequently used in our analysis for MPGA: \StatementNum{1} There exists $\overline{g}_{\mathcal{X}}>0$ such that $\eta(x,y)\leq g(x)\leq \overline{g}_{\mathcal{X}}$ holds for any $(x,y)\in\mathcal{X}\times\mathbb{R}^n$; \StatementNum{2} In view of \cite[Chapter 1, Exercise 7.5(c)]{Clarke:2008BookNonsmooth_analysis_and_control_theory}, there exist some $\mu>0$ and $L>0$ such that $\|\nabla h(u)-\nabla h(v)\|_2\leq L\|u-v\|_2$ holds for any $u,v\in\mathcal{X}_{\mu}$, where $\mathcal{X}_{\mu} := \{z\in\mathbb{R}^n: \dist(z,\mathcal{X})\leq \mu\}$.  
\end{remark}

To discuss that the proposed algorithm MPGA is well-defined, we need the following three technical lemmas. The first two focus on Step 2-Y, and the last one is on Step 2-X.




\begin{lemma}\label{Lemma: Step 2-Y}
Let $(x,y)\in\mathbb{R}^n\times\mathbb{R}^n$, $\alpha>0$, and $y^+(\alpha) = \prox_{\alpha g^*}(y + \alpha x)$. Then the following statements hold:
\begin{enumerate}[label = {\upshape(\roman*)}]
\item $\frac{y-y^+(\alpha)}{\alpha} + x \in \partial g^*(y^+(\alpha))$;
\item $\eta(x,y) + \frac{\|y - y^+(\alpha)\|^2_2}{\alpha} \leq \eta(x,y^+(\alpha))\leq g(x)$;
\item if $(x,y)\in\dom(Q)$, then $(x,y^+(\alpha))\in\dom(Q)$, and
\begin{equation}\label{eq: property of Step 2-Y ObjValue}
	\zeta(x) + \frac{c}{\alpha}\|y - y^+(\alpha)\|^2_2 \leq c \eta(x,y^+(\alpha))
\end{equation}
holds for any $c\in[Q(x,y),+\infty)$.
\end{enumerate}
\end{lemma}
\begin{proof}
By the definition of proximity operators and the convexity of $g^*$, we derive that
$0\in\partial (\alpha g^*)(y^+(\alpha)) + (y^+(\alpha) - y -\alpha x),$ 
and thus obtain Item \StatementNum{1}. From Item \StatementNum{1}, we have 
$g^*(y) \geq g^*(y^+(\alpha)) + \innerP{\frac{1}{\alpha}(y-y^+(\alpha)) + x}{y-y^+(\alpha)} $, which implies the first relation of Item \StatementNum{2}. By invoking the definition of $\eta$, the second relation of Item \StatementNum{2} follows from the \textit{Fenchel-Young Inequality}. 

We finally prove Item \StatementNum{3}.	Suppose $(x,y)\in\dom(Q)$. Then Item \StatementNum{2} and the fact that $\eta(x,y)>0$  leads to $\eta(x,y^+(\alpha)) > 0$. Hence, the assertion $(x,y^+(\alpha))\in\dom(Q)$ follows from $\eta(x,y^+(\alpha)) > 0$ and the fact that $(x,y^+(\alpha))\in\dom(f)\times\dom(g^*)$. By multiplying $c$ on the both sizes of Item \StatementNum{2}, \eqref{eq: property of Step 2-Y ObjValue} is derived upon the fact $\zeta(x) = Q(x,y)\eta(x,y) \leq c\eta(x,y)$.
\end{proof}

\begin{lemma}\label{Lemma: y falls into Y}
Suppose that \Cref{Assumption: X is compact} holds. Let $x\in\mathcal{X}$, $(x,y)\in\dom(Q)$, and $\underline{\alpha}>0 $. For any $\alpha\in[\underline{\alpha},+\infty)$, the vector $y^+(\alpha) = \prox_{\alpha g^*}(y + \alpha x)$ falls into the compact set $\mathcal{Y} := \cup_{z\in\mathcal{X}_{\underline{\alpha}}}\partial g(z)$ with $\mathcal{X}_{\underline{\alpha}} :=\{z\in\mathbb{R}^n: \dist^2(z,\mathcal{X}) \leq\overline{g}_{\mathcal{X}} /\underline{\alpha}  \}$, where $\overline{g}_{\mathcal{X}}>0$ is given by \Cref{Remark: h is global Lips and sup of g} \StatementNum{1}.
\end{lemma}
\begin{proof}
We first prove $y^+(\alpha)\in\mathcal{Y}$. On the one hand, \Cref{Lemma: Step 2-Y} \StatementNum{2} implies that $\frac{1}{\alpha^2}\|y-y^+(\alpha)\|^2_2 \leq g(x)/\alpha \leq g(x)/\underline{\alpha}$ due to  $\eta(x,y)>0$ following from $(x,y)\in\dom(Q)$. On the other hand, from \Cref{Lemma: Step 2-Y} \StatementNum{1}, we derive $w^+ := \frac{1}{\alpha}(y-y^+(\alpha))+x\in\partial g^*(y^+(\alpha))$, or equivalently, $y^+(\alpha)\in\partial g(w^+)$. Therefore, we conclude $y^+(\alpha)\in\mathcal{Y}$ from $y^+(\alpha)\in\partial g(w^+)$, $\|w^+-x\|^2_2 = \frac{1}{\alpha^2}\|y-y^+(\alpha)\|^2_2\leq g(x)/\underline{\alpha}$ and the fact $x\in\mathcal{X}$.

Next, we shall show that $\mathcal{Y}$ is compact. The boundedness of $\mathcal{Y}$ follows from the boundedness of $\mathcal{X}_{\underline{\alpha}}$. We demonstrate the closedness of $\mathcal{Y}$ by verifying that any accumulation point $y^{\star}$ belongs to the set $\mathcal{Y}$. To this end, we let $\{y^k:k\in\mathbb{N} \}\subseteq \mathcal{Y}$ and $\lim_{k\to\infty}y^k = y^{\star}$. Then, there exists $w^k\in \mathcal{X}_{\underline{\alpha}}$ such that $y^k\in\partial g(w^k)$. Due to the compactness of $\mathcal{X}_{\underline{\alpha}}$, there exists a subsequence $\{w^{k_j}:j\in\mathbb{N} \}$ converging to some $w^{\star}\in \mathcal{X}_{\underline{\alpha}}$. Therefore, by letting $j\to\infty$ in $y^{k_j}\in\partial g(w^{k_j})$, we have $y^{\star}\in\partial g(w^{\star})$. Since $w^{\star}\in\mathcal{X}_{\underline{\alpha}}$, we obtain the closedness of $\mathcal{Y}$. This completes the proof.
\end{proof}

Combining the set $\mathcal{X}$ defined in \Cref{Assumption: X is compact} and the set $\mathcal{Y}$ defined in \Cref{Lemma: y falls into Y}, we study the behavior of MPGA within the set $\mathcal{S}$ which is defined by 
\begin{equation}\label{eq: definition of S}
\mathcal{S}:=\{(x,y)\in\mathcal{X}\times\mathcal{Y}:\eta(x,y)>0 \}.
\end{equation}
In particular, one can verify that $\mathcal{S}$ is a bounded subset of $\dom(Q)$ if \Cref{Assumption: X is compact} holds. 
The following lemma concerns the Step 2-X in MPGA restricted to $\mathcal{S}$.



\begin{lemma}\label{proposition: Step 2-X line-search}
Suppose that \Cref{Assumption: X is compact} holds. For $(x,y)\in \mathcal{S}$ satisfying $Q(x,y)\leq F(x^{(0)})$, let $x^{+}_i(\alpha) \in \prox_{\alpha f_i}\left(x_i-\alpha\nabla_i h(x)+\alpha Q(x,y) y_i\right)$, and $x^{+,i}(\alpha)=[x_1;...;x_{i-1};x^{+}_i(\alpha);x_{i+1};...;x_N]  $.  Then the following statements hold:
\begin{enumerate}[label = {\upshape(\roman*)}]
\item For any $\sigma>0$, there exists $\underline{\alpha}_{\mathcal{S},\sigma}>0$, such that 
\begin{equation}\label{eq: proposition: Step 2-X line-search}
	\zeta(x^{+,i}(\alpha)) + \frac{\sigma}{2}\|x^{+,i}(\alpha) - x\|^2_2 \leq c \eta(x^{+,i}(\alpha),y)
\end{equation}
holds for any $(x,y)\in \mathcal{S}$ with $Q(x,y)\leq F(x^{(0)})$, $\alpha\in(0,\underline{\alpha}_{\mathcal{S},\sigma})$, $i\in\mathbb{N}_N$ and $c\in[Q(x,y),+\infty)$;
\item If $x^{+,i}(\alpha)$ satisfies \eqref{eq: proposition: Step 2-X line-search} for some $c\in[0,F(x^{(0)})]$, then $(x^{+,i}(\alpha),y)\in\mathcal{S}$.
\end{enumerate}
\end{lemma}

\begin{proof}
We first prove Item \StatementNum{1}.
It is the result of the proximity operator that, for any $i\in\mathbb{N}_N$, there holds
\begin{equation*}
f_i(x^+_i(\alpha)) + \frac{\|x^+_i(\alpha)-x_i\|^2_2}{2\alpha}
+\innerP{x^+_i(\alpha)-x_i}{ \nabla_i h(x)- Q(x,y)y_i} -  f_i(x_i)\leq 0.
\end{equation*}
This implies
\begin{equation}\label{eq: z1113 1720}
f(x^{+,i}(\alpha)) + \frac{\|x^{+,i}(\alpha)-x\|^2_2}{2 \alpha}
+\innerP{x^{+,i}(\alpha)-x}{ \nabla h(x)- Q(x,y)y} -  f(x)\leq 0.
\end{equation}
Utilizing the \textit{Cauchy–Schwarz Inequality}, we further obtain 
\begin{equation*} 
\frac{\|x^{+,i}(\alpha)-x\|^2_2}{2 \alpha} - \|\nabla h(x)-Q(x,y)y\|_2\|x^{+,i}(\alpha)-x\|_2 + f(x^{+,i}(\alpha))-f(x) \leq 0,
\end{equation*}
which leads to
\begin{equation}\label{eq: z1113 1731}
\begin{split}
	\|x^{+,i}(\alpha)-x\|_2 &\leq \alpha\|\nabla h(x)-Q(x,y)y\|_2 \\
	&\indent +\sqrt{\alpha^2 \|\nabla h(x)-Q(x,y)y\|_2^2 + 2\alpha(f(x) - f(x^{+,i}(\alpha)))}.
\end{split}
\end{equation}
On the right side of \eqref{eq: z1113 1731}, the term $\|\nabla h(x)-Q(x,y)y\|_2$ is bounded for any $(x,y)\in\mathcal{X}\times\mathcal{Y}$ due to the continuity of $\nabla h$ on $\mathcal{X}$, the compactness of $\mathcal{X}\times\mathcal{Y}$ and $Q(x,y)\leq F(x^{(0)})$, Besides, the term $(f(x) - f(x^{+,i}(\alpha))$ is bounded above, according to any $x\in\mathcal{X}$ and $\alpha>0$, since $f$ is continuous on the compact $\mathcal{X}\subseteq \dom(f)$ and bounded below over $\mathbb{R}^n$ (see \Cref{assumption0}). It means that the value of $\|x^{+,i}(\alpha)-x\|_2$ can be narrowed down by the positive scalar $\alpha$ with the help of \eqref{eq: z1113 1731}. Therefore, recalling the statement in \Cref{Remark: h is global Lips and sup of g} \StatementNum{2} that $h$ is globally Lipschitz differentiable on $\mathcal{X}_{\mu}$ with the modulus $L>0$,
we deduce from \eqref{eq: z1113 1731} that there exists some small enough  $\underline{\alpha}_{\mathcal{S},\sigma}\in(0,\frac{1}{\sigma+L})$ such that $\|x^{+,i}(\alpha)-x\|_2\leq\mu/2 $ for any $\alpha\in(0,\underline{\alpha}_{\mathcal{S},\sigma})$, and hence we have 
\begin{equation}\label{eq: z0603 3}
\alpha h(x^{+,i}(\alpha))\leq \alpha\left(h(x) + \innerP{\nabla h(x)}{x^{+,i}(\alpha) - x} + \frac{L}{2}\|x^{+,i}(\alpha)-x\|^2_2\right)
\end{equation}
for any $\alpha\in(0,\underline{\alpha}_{\mathcal{S},\sigma})$. Combining with \eqref{eq: z1113 1720} and \eqref{eq: z0603 3} and rearranging terms, we get
\begin{equation}\label{eq: z1113 1813}
\begin{split}
	\alpha\zeta(x^{+,i}(\alpha))+\frac{1-\alpha L}{2}\|x^{+,i}(\alpha)-x\|^2_2
	& \leq
	\alpha\zeta(x)+\alpha\innerP{x^{+,i}(\alpha)-x}{Q(x,y)y}\\
	& =\alpha Q(x,y) \eta(x^{+,i}(\alpha),y)
\end{split}
\end{equation}
for any $\alpha\in(0,\underline{\alpha}_{\mathcal{S},\sigma})$. Due to the non-negative of the left side of \eqref{eq: z1113 1813} and $Q\geq 0$, the equation \eqref{eq: z1113 1813} implies $\eta(x^{+,i}(\alpha),y)\geq 0$. Therefore, we deduce Item \StatementNum{1} by multiplying $1/\alpha$ on the both sides of  \eqref{eq: z1113 1813} and noting that $\alpha< \underline{\alpha}_{\mathcal{S},\sigma} < \frac{1}{\sigma+L}$ and $Q(x,y) \leq c$.

We next prove Item \StatementNum{2}. For the case $\|x^{+,i}(\alpha) - x\|_2 = 0$, Item \StatementNum{2} is derived from $(x^{+,i}(\alpha),y) = (x,y)\in\mathcal{S} $. For the case $\|x^{+,i}(\alpha) - x\|_2>0$, the relation \eqref{eq: proposition: Step 2-X line-search} implies $\eta(x^{+,i}(\alpha),y)>0$ due to $\zeta\geq 0$ and $c\geq 0$. By dividing by $\eta(x^{+,i}(\alpha),y)$ on the both sides of \eqref{eq: proposition: Step 2-X line-search}, we have $c\geq \zeta(x^{+,i}(\alpha)) / \eta(x^{+,i}(\alpha),y) = Q(x^{+,i}(\alpha),y)$, and then $c\geq F(x^{+,i}(\alpha))$ follows from the relation $Q(x^{+,i}(\alpha),y)\geq F(x^{+,i}(\alpha))$, which follows from the \textit{Fenchel-Young Inequality} and the definitions of $Q$ and $F$. Hence, when $c\leq F(x^{(0)})$, we conclude that   $x^{+,i}(\alpha)\in\mathcal{X} $, and derive $(x^{+,i}(\alpha),y)\in\mathcal{S}$ from the definition of $\mathcal{S}$ (see \eqref{eq: definition of S}). 
\end{proof}

Now we are ready to demonstrate that MPGA is the well-defined.


\begin{theorem}\label{theorem: MPGA is well-defined}
Suppose that \Cref{Assumption: X is compact} holds. Then the following statements hold:
\begin{enumerate}[label = {\upshape(\roman*)}, leftmargin=0.8cm, itemindent=0cm]
\item Step 2-X terminates in finite steps at some $\alpha\geq \underline{\alpha}_{\star}$, where $\underline{\alpha}_{\star} := \min\{\underline{\alpha},\underline{\alpha}_{\mathcal{S},\sigma} \}\gamma$ with $\underline{\alpha}_{\mathcal{S},\sigma}>0$ given by \Cref{proposition: Step 2-X line-search} \StatementNum{1}. In other words, $\alpha_{(t)}\geq \underline{\alpha}_{\star}$ holds for any $t\in\mathbb{N}$;
\item The sequence $\{(x^{(t)},y^{(t)}):t\in\mathbb{N} \}$ generated by MPGA falls into $\mathcal{S}$. For any $t\in\mathbb{N}$, we have
\begin{align}
	&\zeta (x^{(t+1)})+ C_0
	(\|x^{(t+1)} - x^{(t)}\|^2_2 + Q_{(l(t))}\|y^{(t+1)}-y^{(t)}\|^2_2)  \leq Q_{(l(t))}\eta_{(t+1)} ,\label{eq: descent lemma} \\
	&Q(x^{(t+1)},y^{(t+1)})\leq Q_{(l(t+1))} \leq Q_{(l(t))}\leq F(x^{(0)}) ,\label{eq: descent lemma2}
\end{align}
where $C_0:=\min\{\sigma/2,1/\overline{\alpha} \}>0$ and $\eta_{(t+1)} := \eta(x^{(t+1)},y^{(t+1)})$.
\end{enumerate}
\end{theorem}
\begin{proof}
We prove this theorem by induction. It is straightforward to verify that $(x^{(0)},y^{(0)})\in\mathcal{S}$ and $Q_{(l(0))} = Q(x^{(0)},y^{(0)}) = F(x^{(0)})$. We shall show that Item \StatementNum{1}, $(x^{(t+1)},y^{(t+1)})\in\mathcal{S}$, \eqref{eq: descent lemma} and \eqref{eq: descent lemma2} hold under the inductive hypothesis that $(x^{(t)},y^{(t)})\in\mathcal{S}$ and $Q_{(l(t))}\leq  F(x^{(0)})$ in the $t$-th iteration.

We first prove Item \StatementNum{1}. The scalar $\alpha_{(t)}\in [\underline{\alpha} , \overline{\alpha}]$ directly follows from Step 2-Y in MPGA when $i_{(t)} = 0$, and $\alpha_{(t)}\geq \min\{\underline{\alpha},\underline{\alpha}_{\mathcal{S},\sigma} \}\gamma$ is derived from  \Cref{proposition: Step 2-X line-search} \StatementNum{1} when Step 2-X is implemented with $i_{(t)} \in \mathbb{N}_N$.


We next prove $(x^{(t+1)},y^{(t+1)})\in\mathcal{S}$ and \eqref{eq: descent lemma}. Similarly, we consider cases $i_{(t)}=0$ and $i_{(t)}\in\mathbb{N}_N$ separately. For the case when $i_{(t)} = 0$, the Step 2-Y of MPGA yields $x^{(t+1)} = x^{(t)}$ and $y^{(t+1)} = \prox_{\alpha_{(t)} g^*}(y^{(t)}+\alpha_{(t)}x^{(t)})$. Then, invoking \Cref{Lemma: Step 2-Y} \StatementNum{3}, we deduce that $(x^{(t)},y^{(t+1)})\in\dom(Q)$ and \eqref{eq: descent lemma} holds. 	
Due to the fact $x^{(t+1)} = x^{(t)}$, the hypothesis $(x^{(t)},y^{(t)})\in\mathcal{S}$ results in $x^{(t+1)}\in\mathcal{X}$, and the assertion $(x^{(t)},y^{(t+1)})\in\dom(Q)$ leads to $\eta(x^{(t+1)},y^{(t+1)})>0$. 	
Together with $y^{(t+1)}\in\mathcal{Y}$, which follows from \Cref{Lemma: y falls into Y}, we conclude that  $(x^{(t+1)},y^{(t+1)}) \in\mathcal{S}$ by the definition of $\mathcal{S}$ (see \eqref{eq: definition of S}). For the case when $i\in\mathbb{N}_N$, we can verify that the relation 
\eqref{eq: descent lemma} holds with the help of \Cref{proposition: Step 2-X line-search} \StatementNum{1} and the fact $y^{(t+1)} = y^{(t)}$. Besides, the assertion $(x^{(t+1)},y^{(t+1)}) \in\mathcal{S}$ follows from \Cref{proposition: Step 2-X line-search} \StatementNum{2}.

We finally prove \eqref{eq: descent lemma2} upon the assertion $(x^{(t+1)},y^{(t+1)})\in\mathcal{S}$. It suffices to show that the second relation holds in \eqref{eq: descent lemma2}, wherein the first and third relations follow from the definition of $Q_{(l(t+1))}$ and the hypothesis $Q_{(l(t))}\leq  F(x^{(0)})$, respectively.	Let $Q_{(t+1)}:= Q(x^{(t+1)},y^{(t+1)})$. Dividing by $\eta_{(t+1)}$ on the both sides of \eqref{eq: descent lemma}, we deduce $Q_{(t+1)} = \zeta(x^{(t+1)}) / \eta_{(t+1)} \leq Q_{(l(t))}$. Then, we have $Q_{(l(t+1))} \leq \max\{Q_{(t+1)},Q_{(l(t))} \}
\leq Q_{(l(t))}$, and hence the second relation of  \eqref{eq: descent lemma2} holds. 
\end{proof}

At the rest of this section, we establish some basic convergence results of MPGA. To this end, we first show the continuity of the extended objective $Q$ on the set $\mathcal{S}$ in the following lemma.
\begin{lemma}\label{Lemma: eta is Lipschitz continuous}
Suppose that \Cref{Assumption: X is compact} holds. Let $\eta:\mathbb{R}^n\times\mathbb{R}^n \to [-\infty,+\infty)$ be defined at $(x,y)\in \mathbb{R}^n\times\mathbb{R}^n$ as $\eta(x,y) = \innerP{x}{y}-g^*(y)$. Then $\eta$ is Lipschitz continuous on $\mathcal{X}\times\mathcal{Y}$, and thus the objective $Q$ given by \eqref{definition: function Q} is continuous on $\mathcal{S}$.
\end{lemma}
\begin{proof}
Due to the compactness of $\mathcal{X}\times\mathcal{Y}$, it suffices to show that $g^*$ is Lipschitz continuous on $\mathcal{Y}$. Let $y_1,y_2\in\mathcal{Y}$. According to the definition of $\mathcal{Y}$ (see \Cref{Lemma: y falls into Y}), there exists some $w_i\in\partial g^*(y_i)$ satisfying $\dist^2(w_i,\mathcal{X})\leq \underline{\alpha}^{-1}\overline{g}_{\mathcal{X}}$ for any $i=1,2$. By the convexity of $g^*$, there holds that $g^*(y_2)\geq g^*(y_1) + \innerP{w_1}{y_2-y_1}$ and $g^*(y_1)\geq g^*(y_2) + \innerP{w_2}{y_1-y_2}$.
Together with these two relations, we have
\begin{equation*}
\begin{split}
	|g^*(y_1)- g^*(y_2)|&\leq \max\{ \|w_1\|_2,\|w_2\|_2 \} \|y_1-y_2\|_2\\ &\leq \sup\left\{ \|x\|_2+\sqrt{\underline{\alpha}^{-1}\overline{g}_{\mathcal{X}}}:x\in\mathcal{X} \right\} \|y_1-y_2\|_2.
\end{split}
\end{equation*}
This completes the proof by the boundedness of $\mathcal{X}$.
\end{proof}

We finally present the convergence results of MPGA in the following theorem.
\begin{theorem}\label{Lemma: d to 0}
Suppose that \Cref{Assumption: X is compact} holds and the sequence $\{(x^{(t)},y^{(t)}):t\in\mathbb{N} \}$ be generated by MPGA with some $\mathcal{I}_{\infty} = (i_{(0)},i_{(1)},i_{(2)},...)$. Then $\{Q_{(l(t))}:t\in\mathbb{N} \}$ descends monotonically, and there holds
\begin{equation}\label{eq: Qlt converges to Q_I_infinity}
\lim_{t\to\infty} Q_{(l(t))} = Q_{\mathcal{I}_{\infty}}
\end{equation}
for some $Q_{\mathcal{I}_{\infty}}\in[0,F(x^{(0)})]$. Furthermore, there holds
\begin{equation}\label{eq: d to 0}
\lim\limits_{t\to\infty} \|x^{(t+1)}-x^{(t)}\|^2_2 + Q_{\mathcal{I}_{\infty}}\|y^{(t+1)}-y^{(t)}\|^2_2 = 0.
\end{equation}
\end{theorem}

\begin{proof}
The relation \eqref{eq: descent lemma2} reveals that $\{Q_{(l(t))}:t\in\mathbb{N} \}$ descends monotonically and $Q_{(l(t))}\geq Q(x^{(t+1)},y^{(t+1)})$ holds for any $t\in\mathbb{N}$. Then, the relation \eqref{eq: Qlt converges to Q_I_infinity} follows from the fact $Q\geq 0$.
We next dedicate to proving \eqref{eq: d to 0}.	It is straightforward to verify \eqref{eq: d to 0} by passing to the limit on the both sides of \eqref{eq: descent lemma} when $Q_{\mathcal{I}_{\infty}}=0$. Therefore, we suppose $Q_{\mathcal{I}_{\infty}}>0$ in the rest of the proof. Let $d_{(t+1)} := \|x^{(t+1)} - x^{(t)}\|^2_2 + Q_{(l(t))}\|y^{(t+1)}-y^{(t)}\|^2_2$ for $t\in\mathbb{N}$. Then, it is derived from \eqref{eq: descent lemma} that
\begin{equation}\label{eq: z1130 1557}
C_0 d_{(t+1)} \leq (Q_{(l(t))} - Q_{(t+1)}) \eta_{(t+1)}.
\end{equation}
Notice that $d_{(t+1)}\in \{d_{(l(t+M+1)-j)}:j\in\mathbb{N}_M  \}$ holds for any $t\in\mathbb{N}$. We shall show that $\lim_{t\to\infty} d_{(l(t+M+1)-j)} = 0$ holds for any $j\in\mathbb{N}_M$ by induction.

First, for the case when $j=0$, replacing $t$ by $l(t+M+1)-1$ in \eqref{eq: z1130 1557}, we obtain
\begin{equation}\label{eq: z0928 311}
C_0 d_{(l(t+M+1))} 
\leq \left(  Q_{(l(t))} -  Q_{(l(t+M+1))} \right) \overline{g}_{\mathcal{X}},
\end{equation}
where $\overline{g}_{\mathcal{X}}>0$ is given by \Cref{Remark: h is global Lips and sup of g} \StatementNum{1}.	Owing to \eqref{eq: Qlt converges to Q_I_infinity}, we deduce $\lim_{t\to\infty} d_{(l(t+M+1))} = 0$ by passing to the limit on \eqref{eq: z0928 311} with $t\to\infty$.
Second, under the inductive hypothesis that $\lim_{t\to\infty} d_{(l(t+M+1)-j)} = 0$ holds for any $j\in\mathbb{N}_J$ with some $J<M$, we shall show $\lim_{t\to\infty} d_{( l(t+M+1)-(J+1))} = 0$ by contradiction. If $\epsilon:= \mathop{\lim\sup}_{t\to\infty} d_{( l(t+M+1)-(J+1))} > 0$, there exists some subsequence $\mathcal{K}\subseteq \mathbb{N}$ such that
\begin{equation}\label{eq: z1130 1622}
\lim_{t\in\mathcal{K}\to\infty} d_{( l(t+M+1)-(J+1))} = \epsilon >0.
\end{equation}
Since \Cref{theorem: MPGA is well-defined} \StatementNum{2} shows that $\{(x^{(t)},y^{(t)}):t\in\mathbb{N}  \}$ falls into the compact set $\mathcal{X}\times\mathcal{Y}$, there exists some subsequence $\mathcal{K}_2\subseteq \mathcal{K}$ such that
\begin{equation}\label{eq: z0716 4}
\lim_{t\in\mathcal{K}_2\to\infty} \left( x^{( l(t+M+1)-(J+1))},y^{( l(t+M+1)-(J+1))} \right) = \left( \mathring{x}^{J+1},\mathring{y}^{J+1} \right)
\end{equation}
holds for some $\left( \mathring{x}^{J+1},\mathring{y}^{J+1} \right)\in \mathcal{X}\times\mathcal{Y}$. Replacing $t$ by $l(t+M+1)-(J+1)-1$ in \eqref{eq: z1130 1557} and passing to the limit with $t\in\mathcal{K}_2\to\infty$, we deduce $\eta\left( \mathring{x}^{J+1},\mathring{y}^{J+1} \right)>0$ upon the fact that $\eta$ is continuous  on $\mathcal{X}\times\mathcal{Y}$ (see \Cref{Lemma: eta is Lipschitz continuous}), and hence we have $\left( \mathring{x}^{J+1},\mathring{y}^{J+1} \right)\in\mathcal{S}$. Besides, in view of the induction hypothesis and $Q_{\mathcal{I}_{\infty}}>0$, we derive from \eqref{eq: z0716 4} that 
\begin{equation}\label{eq: z1130 1616}
\lim_{t\in\mathcal{K}_2\to\infty} \left( x^{( l(t+M+1))},y^{( l(t+M+1))} \right) = \left( \mathring{x}^{J+1},\mathring{y}^{J+1} \right).
\end{equation}
Due to the continuity of $Q$ on $\mathcal{S}$ (see \Cref{Lemma: eta is Lipschitz continuous}), we finally obtain from \eqref{eq: z0716 4} and \eqref{eq: z1130 1616} that
\begin{equation}\label{eq: 1130 1621}
\lim_{t\in\mathcal{K}_2\to\infty} Q_{( l(t+M+1)-(J+1))}
= Q\left( \mathring{x}^{J+1},\mathring{y}^{J+1} \right)
= \lim_{t\in\mathcal{K}_2\to\infty} Q_{( l(t+M+1))}
= Q_{\mathcal{I}_{\infty}}.
\end{equation}
Notice that \eqref{eq: z1130 1557} and \eqref{eq: 1130 1621} imply 
$$\lim_{t\in\mathcal{K}_2\to\infty} C_0 d_{( l(t+M+1)-(J+1))} \leq \lim_{t\in\mathcal{K}_2\to\infty} (Q_{(l(t+M+1))} - Q_{( l(t+M+1)-(J+1))}) \overline{g}_{\mathcal{X}} = 0,$$
which contradicts \eqref{eq: z1130 1622}. 
Consequently, $\lim_{t\to\infty} d_{(l(t+M+1)-j)} = 0$ holds for any $j\in\mathbb{N}_M$, and $\lim_{t\to\infty} d_{(t)} = 0$ follows immediately. This completes the proof.
\end{proof}


\section{Subsequential convergence analysis of CMPGA and RMPGA} \label{section: Subsequential convergence analysis}

In this section, we investigate the subsequential convergence of two specific algorithms within the framework of MPGA, namely, cyclic MPGA (CMPGA) and randomized MPGA (RMPGA), where the update block at each iteration are chosen in cyclic and randomized fashions respectively. The first subsection concerns the subsequential convergence of CMPGA, while the second subsection regards that of RMPGA.

\subsection{Subsequential convergece of CMPGA} \label{section: MPGAc}

This subsection is devoted to the subsequential convergece analysis of CMPGA. For this algorithm, the index $i$ of Step 1 at the $t$-th iteration is chosen by $i:=t \mod (N+1)$, i.e., the remainder of $t$ divided by $N+1$. Hence, every $N+1$ iterations consist of an epoch. For ease of presentation, we shall convert the corner-marks of the sequence $\{(x^{(t)},y^{(t)}):t\in\mathbb{N} \}$ generated by CMPGA into some $(k,i)\in\mathbb{N}\times\mathbb{N}_N^0$ via $k:=\lfloor t/(N+1) \rfloor$ and $i:=t\mod(N+1)$, i.e.,
\begin{equation}\label{eq: z0504 2037}
	\left(x^{(k,i)},y^{(k,i)}\right) = \left(x^{(k(N+1)+i)},y^{(k(N+1)+i)} \right).
\end{equation} 
With these new corner-marks, we now give a brief description on an epoch of CMPGA. Given $\left(x^{(k,0)},y^{(k,0)} \right)\in\mathbb{R}^n\times\mathbb{R}^n$, CMPGA performs the Step 2-Y at the first iteration of the $k$-th epoch to generate $y^{(k,1)}$ as
\begin{equation}\label{eq: Step 2-Y for MPGAc}
	y^{(k,1)} = \prox_{\alpha_{(k,0)} g^*}(y^{(k,0)}+\alpha_{(k,0)} x^{(k,0)})
\end{equation}
with some $\alpha_{(k,0)}\in[\underline{\alpha},\overline{\alpha}]$, while keeping $x^{(k,1)} = x^{(k,0)}$. Then, at the $(i+1)$-th iteration of an epoch, $i\in\mathbb{N}_N$, CMPGA keeps $y^{(k,i+1)} = y^{(k,i)}$ and conduct the Step 2-X to produce the new iteration $x^{(k,i+1)}$ by updating only the $i$-th block of $x^{(k,i)}$ as
\begin{equation}\label{eq: Step 2-X for MPGAc}
	x^{(k,i+1)}_i 
	\in \prox_{\alpha_{(k,i)} f_i}\left(x^{(k,i)}_i-\alpha_{(k,i)}\nabla_i h(x^{(k,i)})+\alpha_{(k,i)} Q(x^{(k,i)},y^{(k,i)}) y^{(k,i)}_i\right)
\end{equation}
with some $\alpha_{(k,i)}\in[\underline{\alpha}_{\star},\overline{\alpha}]$. Overall, in the $k$-th epoch of CMPGA, for any $i\in\mathbb{N}_N$, we have
\begin{equation}\label{eq: CMPGA}
	y^{(k,i)} = y^{(k+1,0)} \text{~~~~and~~~~}
	x_j^{(k,i)} = \begin{cases}
		x_j^{(k,0)}, & \text{if } j\geq i,\\
		x_j^{(k+1,0)}, & \text{if } j< i.
	\end{cases}
\end{equation}

The following lemma is the direct result of \Cref{theorem: MPGA is well-defined} \StatementNum{2} and \Cref{Lemma: d to 0}.
\begin{lemma}\label{lemma: MPGAc d to 0}
	Suppose that \Cref{Assumption: X is compact} holds.
	Let the sequence $\{(x^{(t)},y^{(t)}):t\in\mathbb{N} \}$ be generated by CMPGA, and $ \left\{ \left(x^{(k,i)},y^{(k,i)} \right): (k,i)\in\mathbb{N}\times\mathbb{N}_N^0\right\}$ be defined by \eqref{eq: z0504 2037}. Then, the following statements hold:
	\begin{enumerate}[label = {\upshape(\roman*)}]
		\item For any $(k,i)\in\mathbb{N}\times\mathbb{N}$, there holds $\left(x^{(k,i)},y^{(k,i)} \right) \in\mathcal{S} $ and
		\begin{multline}\label{eq: descent lemma for MPGAc}
			\zeta(x^{(k,i+1)})
			+C_0
			(\|x^{(k,i+1)} - x^{(k,i)}\|^2_2 + Q_{(l(k(N+1)+i))}\|y^{(k,i+1)}-y^{(k,i)}\|^2_2)\\
			\leq Q_{(l(k(N+1)+i))} \eta(x^{(k,i+1)},y^{(k,i+1)});
		\end{multline}
		\item There exists some $Q_{c_{\infty}}\in[0,F(x^{(0)})]$ such that $$\lim_{k\to\infty} Q_{(l(k(N+1)+i))} = Q_{c_{\infty}}$$ holds for any $i\in\mathbb{N}_N^0$;
		\item Furthermore, there holds
		\begin{equation}\label{eq: d to 0 for MPGAc}
			\lim\limits_{k\to\infty} \|x^{(k,i+1)}-x^{(k,i)}\|^2_2 + Q_{c_{\infty}} \|y^{(k,i+1)}-y^{(k,i)}\|^2_2 = 0
		\end{equation}
		for any $i\in\mathbb{N}$.
	\end{enumerate}
\end{lemma}

Now we are ready to establish the subsequential convergence for CMPGA.

\begin{theorem}\label{theorem: subsequential convergent of MPGAc}
	Suppose \Cref{Assumption: X is compact} holds. Let $(x^{\star},y^{\star})\in\mathbb{R}^n\times\mathbb{R}^n$ be an accumulation point of $\{(x^{(t)},y^{(t)}):t\in\mathbb{N} \}$ generated by CMPGA. Then the following statements hold:
	\begin{enumerate}[label = {\upshape(\roman*)}]
		\item $(x^{\star},y^{\star})\in\mathcal{S}$ and $\underline{\eta_c}:=\inf\{\eta(x^{(t)},y^{(t)}):t\in\mathbb{N} \}>0$;
		\item $x^{\star}\in\dom(F)$ is a critical point of $F$ with $F(x^{\star})y^{\star}\in F(x^{\star})\partial g(x^{\star})$.
	\end{enumerate}
\end{theorem}

\begin{proof}
	Let $\left\{ \left(x^{(k,i)},y^{(k,i)}\right):(k,i)\in\mathbb{N}\times\mathbb{N}_N^0 \right\}$ be defined by \eqref{eq: z0504 2037}. In view of the \textit{Pigeonhole Principle}, there exists a subsequence  $\mathcal{K}\subseteq\mathbb{N}$ such that $\{(x^{(k,\bar{i})},y^{(k,\bar{i})}): k\in\mathcal{K} \}$ with a uniform $\bar{i}\in\mathbb{N}_N^0$ converges to $(x^{\star},y^{\star})$, i.e.,
	\begin{equation}\label{eq: z1024 0223}
		\lim_{k\in\mathcal{K}\to\infty}(x^{(k,\bar{i})},y^{(k,\bar{i})})= (x^{\star},y^{\star}).
	\end{equation}
	
	We first prove $(x^{\star},y^{\star})\in\mathcal{S}$ in Item \StatementNum{1}. According to the definition of $\mathcal{S}$ (see \eqref{eq: definition of S}), it suffices to prove $\eta(x^{\star},y^{\star})>0$, since $(x^{\star},y^{\star})\in\mathcal{X}\times\mathcal{Y}$ follows from the compactness of $\mathcal{X}$ and $\mathcal{Y}$ immediately. We shall prove $\eta(x^{\star},y^{\star})>0$ by contradiction. Suppose $\eta(x^{\star},y^{\star})\leq 0$. Notice that the relation \eqref{eq: CMPGA} indicates that $y^{(k,1)} = y^{(k,i+1)}$ holds for any $i\in\mathbb{N}_N$, and \Cref{lemma: MPGAc d to 0} \StatementNum{3} implies that $\lim_{k\to\infty} \|x^{(k,i)}-x^{(k,0)}\| = 0 $ holds for any $i\in\mathbb{N}_N^0$. Then, the equation \eqref{eq: z1024 0223} implies
	\begin{equation}\label{eq: z0317 1544}
		\lim_{k\in\mathcal{K}\to\infty} \left(x^{(k,0)},y^{(k,1)}  \right) = (x^{\star},y^{\star})
	\end{equation}
	Owing to \eqref{eq: Step 2-Y for MPGAc}, there holds for any $k\in\mathbb{N}$ that
	\begin{equation}\label{eq: z1209 1703}
		\dist(x^{(k,0)},\partial g^*(y^{(k,1)}))
		\leq \frac{1}{\alpha_{(k,0)}} \|y^{(k,0)} - y^{(k,1)} \|_2  
		\leq \sqrt{\frac{1}{\underline{\alpha}}\eta(x^{(k,0)},y^{(k,1)})},
	\end{equation}
	where the first relation follows from \Cref{Lemma: Step 2-Y} \StatementNum{1}, and the second relation follows from \Cref{Lemma: Step 2-Y} \StatementNum{2}, $\eta(x^{(k,0)},y^{(k,0)})>0$ and $\alpha_{(k,0)}\geq \underline{\alpha}$. From \eqref{eq: z1209 1703}, we deduce that, for any $\epsilon>0$, there exists some $z_{\epsilon,k} \in \partial g^*(y^{(k,1)})$ such that
	\begin{equation}\label{eq: z0317 1605}
		\|x^{(k,0)} - z_{\epsilon,k}\|_2 \leq \sqrt{\frac{1}{\underline{\alpha}}\eta(x^{(k,0)},y^{(k,1)})} + \epsilon,
	\end{equation}
	which implies the boundedness of $\{z_{\epsilon,k}:k\in\mathcal{K} \}$ owing to the boundedness of $\mathcal{X}$, which is the supset of $\{x^{(k,0)}:k\in\mathcal{K} \}$, and $\eta(x^{(k,0)},y^{(k,1)})\leq \overline{g}_{\mathcal{X}} $ given by \Cref{Remark: h is global Lips and sup of g} \StatementNum{1}. Inspired of this, we further assume the subsequence $\mathcal{K}_2\subseteq\mathcal{K}$ so that $\{z_{\epsilon,k}:k\in\mathcal{K}_2 \}$ converges to some $z_{\epsilon,\infty}$. Invoking the continuity of $\eta$ on $\mathcal{X}\times\mathcal{Y}$ by \Cref{Lemma: eta is Lipschitz continuous}, we derive $z_{\epsilon,\infty}\in\partial g^*(y^{\star})$ from the definition of limiting subdifferentials due to $\{y^{(k,1)}:k\in\mathcal{K}_2\}\to y^{\star}$ by \eqref{eq: z0317 1544} and $z_{\epsilon,k}\in \partial g^*(y^{(k,1)})$ for any $k\in\mathcal{K}_2$. Passing to the limit on the both sides of \eqref{eq: z0317 1605} with $k\in\mathcal{K}_2 \to\infty$, we obtain $\|x^{\star} - z_{\epsilon,\infty}\|_2 \leq \epsilon$ from the continuity of $\eta$ on $\mathcal{X}\times\mathcal{Y}$ by \Cref{Lemma: eta is Lipschitz continuous}, \eqref{eq: z0317 1544} and $\eta(x^{\star},y^{\star})\leq 0$. Due to the arbitrariness of $\epsilon$, we deduce $\dist(x^{\star},\partial g^*(y^{\star})) = 0$, or equally, $x^{\star}\in\partial g^*(y^{\star})$ owing to the closedness of $\partial g^*(y^{\star})$. Consequently, the assumption $\eta(x^{\star},y^{\star})\leq 0$ results in $x^{\star}\in\partial g^*(y^{\star})$ and leads to $g(x^{\star}) = 0$. However, $g(x^{\star}) \neq  0$ follows from the fact $x^{\star}\in\mathcal{X}$. Therefore, we deduce $\eta(x^{\star},y^{\star})>0$ by the contradiction, and $(x^{\star},y^{\star})\in\mathcal{S}$ follows immediately. Besides, $\underline{\eta_c}>0$ is derived since the statement $\{(x^{(t)},y^{(t)}):t\in\mathbb{N} \}\in\mathcal{S}$ in \Cref{theorem: MPGA is well-defined} \StatementNum{2} indicates that $ \eta(x^{(t)},y^{(t)})>0$ holds for each $t\in\mathbb{N}$, and $ \eta(x^{\star},y^{\star})>0$ holds for each $(x^{\star},y^{\star})$ being an accumulation point of $\{(x^{(t)},y^{(t)}):t\in\mathbb{N} \}$. This completes the proof of Item \StatementNum{1}.
	
	We next prove Item \StatementNum{2}. The statement that $x^{\star}\in\dom(F)$ in Item \StatementNum{2} is the direct result of $(x^{\star},y^{\star})\in\mathcal{S}$ given by Item \StatementNum{1}, and $g(x^{\star})\neq 0$ follows from $x^{\star}\in\dom(F)$ and the definition of $F$ (see \eqref{definition: function F}). When $\zeta(x^{\star}) = 0$, Item \StatementNum{2} comes immediately since $x^{\star}$ is a global minimizer of $F$ with $F(x^{\star}) = \zeta(x^{\star}) / g(x^{\star}) = 0$. We next focus on the case when $\zeta(x^{\star}) >0$. Replacing $i$ with $\bar{i}-1$ in \eqref{eq: descent lemma for MPGAc}, we deduce that
	\begin{equation}\label{eq: z0327 1644}
		\zeta(x^{(k,\bar{i})})
		\leq  Q_{(l(k(N+1)+\bar{i}-1))} \eta(x^{(k,\bar{i})},y^{(k,\bar{i})})
		\leq  Q_{(l(k(N+1)+\bar{i}-1))} \overline{g}_{\mathcal{X}},
	\end{equation}
	where $\overline{g}_{\mathcal{X}}\geq 0$ is given by \Cref{Remark: h is global Lips and sup of g} \StatementNum{1}. Passing to the limit on the both sides of \eqref{eq: z0327 1644} with $k\in\mathcal{K}\to\infty$, we obtain $\zeta(x^{\star})\leq Q_{c_{\infty}} \overline{g}_{\mathcal{X}}$ due to \eqref{eq: z1024 0223}, the continuity of $\zeta$ on $\mathcal{X}$ and \Cref{lemma: MPGAc d to 0} \StatementNum{2}. Thus, $\zeta(x^{\star})>0$ forces $Q_{c_{\infty}}>0$. 
	Owing to $Q_{c_{\infty}}>0$, \eqref{eq: d to 0 for MPGAc}  and \eqref{eq: z1024 0223}, there holds for any $i\in\mathbb{N}_N^0$ that
	\begin{equation}\label{eq: z0327 1643}
		\lim_{k\in\mathcal{K}\to\infty} \left(x^{(k,i)},y^{(k,i)}  \right) = (x^{\star},y^{\star}).
	\end{equation}
	Then, according to CMPGA, we derive from \eqref{eq: Step 2-Y for MPGAc} that
	\begin{multline}\label{eq: z1204 1704}
		\alpha_{(k,0)} g^*(y^{(k,1)}) + \frac{1}{2}\|y^{(k,1)}-y^{(k,0)}-\alpha_{(k,0)} x^{(k,0)}\|^2_2\\
		\leq \alpha_{(k,0)} g^*(y) + \frac{1}{2}\|y-y^{(k,0)}-\alpha_{(k,0)} x^{(k,0)}\|^2_2
	\end{multline}
	holds for any $y\in\mathbb{R}^n$, and \eqref{eq: z1204 1704} yields that
	\begin{multline}\label{eq: z1204 1708}
		g^*(y^{(k,1)}) + \frac{1}{2\overline{\alpha}}\|y^{(k,1)}-y^{(k,0)}\|^2_2 + \innerP{y^{(k,1)}-y^{(k,0)}}{x^{(k,0)}}\\
		\leq 
		g^*(y) + \frac{1}{2\underline{\alpha}}\|y-y^{(k,0)}\|^2_2 + \innerP{y-y^{(k,0)}}{x^{(k,0)}}.
	\end{multline}
	By passing to the limit on the both sides of \eqref{eq: z1204 1708} with $k\in\mathcal{K}\to\infty$, we obtain that $g^*(y^{\star})\leq g^*(y) +  \frac{1}{2\underline{\alpha}}\|y-y^{\star}\|^2_2 + \innerP{y-y^{\star}}{x^{\star}}$ holds for any $y\in\mathbb{R}^n$, which indicates
	\begin{equation}\label{eq: z1204 1714}
		y^{\star}=\prox_{\underline{\alpha} g^*} (y^{\star}+\underline{\alpha} x^{\star}).
	\end{equation}
	Similarly, there holds for any $i\in\mathbb{N}_N$ that
	\begin{equation}\label{eq: z1204 1715}
		x^{\star}_i\in\prox_{\underline{\alpha}_{\star} f_i}(x^{\star}_i-\underline{\alpha}_{\star}\nabla_i h(x^{\star})+\underline{\alpha}_{\star} Q(x^{\star},y^{\star})y^{\star}_i),
	\end{equation}
	owing to \eqref{eq: Step 2-X for MPGAc}, \eqref{eq: z1024 0223}, $\alpha_{(k,i)}\in[\underline{\alpha}_{\star},\overline{\alpha}]$, $(x^{\star},y^{\star})\in\mathcal{S}$ and the continuity of $Q$ on $\mathcal{S}$ (see \Cref{Lemma: eta is Lipschitz continuous}). 
	Together with \eqref{eq: z1204 1714} and \eqref{eq: z1204 1715}, we derive Item \StatementNum{2} with the help of \Cref{prop: stationary point of prox}.
\end{proof}

\subsection{Subsequential convergence of RMPGA}\label{section: MPGAs}

This subsection is devoted to the subsequential convergence of RMPGA. At the $t$-th iteration, RMPGA picks an index $i$ from $\mathbb{N}_N^0$ with probability $p^{(t)}_i\geq p_{\min}>0$, where $\{p^{(t)}_i:i\in\mathbb{N}_N^0 \}$ forms a probability distribution. For convenience, we use $(x^{(t),i},y^{(t),i})$ to denote the point generated by the $t$-th iteration of RMPGA if a specific index $i\in \mathbb{N}_N^0$ is chosen. Given $(x^{(t)},y^{(t)})\in\mathbb{R}^n\times\mathbb{R}^n$, if the index $i=0$ is picked at the $t$-th iteration, then RMPGA keeps $x^{(t+1)} = x^{(t)}$ and performs Step 2-Y to set $y^{(t+1)} = y^{(t),+}$, where 
\begin{equation}\label{eq: Step 2-Y for MPGAs}
	y^{(t),+} = \prox_{\alpha_{(t),0} g^*}(y^{(t)}+\alpha_{(t),0} x^{(t)})
\end{equation}
for some $\alpha_{(t),0}\in[\underline{\alpha},\overline{\alpha}]$. When an index $i\in\mathbb{N}_N$ is chosen at the $t$-th iteration, RMPGA keeps $y^{(t+1)} = y^{(t)} $ while performing Step 2-X to generate the new iterator $x^{(t+1)}$ as
\begin{equation*}
	x^{(t+1)}_j = \begin{cases}
		x^{(t)}_j, & j\neq i,\\
		x^{(t),+}_j, & j=i,
	\end{cases}
\end{equation*}
where
\begin{equation}\label{eq: Step 2-X for MPGAs}
	x^{(t),+}_j \in \prox_{\alpha_{(t),j} f_j}\left(x^{(t)}_j-\alpha_{(t),j}\nabla_j h(x^{(t)})+\alpha_{(t),j} Q(x^{(t)},y^{(t)}) y^{(t)}_j\right)
\end{equation}
with some $\alpha_{(t),j}\in[\underline{\alpha}_{\star},\overline{\alpha}]$. We establish the subsequential convergence of RMPGA in the following theorem.

\begin{theorem}
	Suppose \Cref{Assumption: X is compact} holds. Let $(x^{\star},y^{\star})$ be an accumulation point of $\{(x^{(t)},y^{(t)}):t\in\mathbb{N} \}$ generated by RMPGA. Then $x^{\star}\in\dom(F)$ is a critical point of $F$ with $F(x^{\star})y^{\star}\in F(x^{\star})\partial g(x^{\star})$ almost surely.
\end{theorem}

\begin{proof}
	Invoking \Cref{Lemma: y falls into Y}, we deduce $(x^{\star},y^{\star}) \in \mathcal{X}\times \mathcal{Y}$ by the compactness of $\mathcal{X}\times \mathcal{Y}$. When $\zeta(x^{\star}) = 0$, this theorem follows immediately since $x^{\star} $ is a global minimizer of $F$ with $F(x^{\star}) = 0$. We next focus on the case when $\zeta(x^{\star}) > 0$. Let $\mathcal{K}$ be a subsequence such that $\{(x^{(t)},y^{(t)}):t\in\mathcal{K} \}$ converges to $(x^{\star},y^{\star})$. Then  $\lim_{t\in\mathcal{K}\to\infty} x^{(t+1)} = x^{\star}$ follows from \eqref{eq: d to 0}, and $\zeta(x^{\star})\leq Q_{\mathcal{I}_{\infty}} \eta(x^{\star},y^{\star})$ is derived by passing to the limit with $t\in\mathcal{K}\to\infty$ on the both sides of \eqref{eq: descent lemma} and noting that $\eta$ is continuous on $\mathcal{X}\times \mathcal{Y}$ by \Cref{Lemma: eta is Lipschitz continuous}. Hence, $\zeta(x^{\star})>0$ results in $\eta(x^{\star},y^{\star})>0$ and $Q_{\mathcal{I}_{\infty}}>0$. On the one hand, $\eta(x^{\star},y^{\star})>0$ and $(x^{\star},y^{\star}) \in \mathcal{X}\times \mathcal{Y}$ indicate $(x^{\star},y^{\star}) \in \mathcal{S}$. On the other hand, $Q_{\mathcal{I}_{\infty}}>0$ and \eqref{eq: d to 0} yield that both of $\{\|x^{(t+1)}-x^{(t)}\|:t\in\mathcal{K} \}$ and $\{\|y^{(t+1)}-y^{(t)}\|:t\in\mathcal{K} \}$ converge to zero.  Let $\widehat{x}^{(t)}\in\mathbb{R}^n$ be defined by $\widehat{x}^{(t)}_j = x^{(t),+}_j$ for $j\in\mathbb{N}_N$ and $\widehat{y}^{(t)} = y^{(t),+}$.
	Notice that for any $\xi>0$ and $t\in\mathbb{N}$, there holds in the sense of probability that
	\begin{align*}
		&\Prob{\|x^{(t+1)}-x^{(t)}\|_2^2+\|y^{(t+1)}-y^{(t)}\|_2^2 > \frac{\xi}{(N+1)} } \\
		&= \sum_{i=0}^{N} p_i^{(t)} \Prob{\|x^{(t),i}_i-x^{(t)}_i\|_2^2+\|y^{(t),i}-y^{(t)}\|_2^2 > \frac{\xi}{(N+1)}}\\
		&\geq p_{\min} \Prob{\|\widehat{x}^{(t)} - x^{(t)}\|_2^2+\|\widehat{y}^{(t)} - y^{(t)}\|_2^2 > \xi }.
	\end{align*}
	Therefore, under the prior hypothesis $\zeta(x^{\star})>0$, we deduce that $\{(\widehat{x}^{(t)},\widehat{y}^{(t)}):t\in\mathcal{K} \}$ converges to $(x^{\star},y^{\star})$ in probability. Invoking \cite[Theorem 6.3.1 (b)]{Resnick:A_Probability_Path}, we claim some subsequence $\mathcal{K}_2\subseteq \mathcal{K}$ such that $\{(\widehat{x}^{(t)},\widehat{y}^{(t)}):t\in\mathcal{K}_2 \}$ converges to $(x^{\star},y^{\star})$ almost surely under $\zeta(x^{\star})>0$. By passing to the limit with $t\in\mathcal{K}_2\to\infty$ after replacing $y^{(t),+}$ by $\widehat{y}^{(t)}$ in \eqref{eq: Step 2-Y for MPGAs} and replacing $x^{(t),+}_j$ by $\widehat{x}^{(t)}_j$ in \eqref{eq: Step 2-X for MPGAs}, we deduce that the following two relations hold almost surely under $\zeta(x^{\star})>0$:
	\begin{align*}
		&y^{\star}\in\prox_{\underline{\alpha} g^*} (y^{\star}+\underline{\alpha} x^{\star}),\\
		&x^{\star}_j\in\prox_{\underline{\alpha}_{\star} f_j}(x^{\star}_j-\underline{\alpha}_{\star}\nabla_j h(x^{\star})+\underline{\alpha}_{\star} Q(x^{\star},y^{\star})y^{\star}_j), ~~~~~j\in\mathbb{N}_N,
	\end{align*}
	owing to $\alpha_{(t),0}\in[\underline{\alpha},\overline{\alpha}]$, $\alpha_{(t),j}\in[\underline{\alpha}_{\star},\overline{\alpha}]$, $(x^{\star},y^{\star})\in\mathcal{S}$ and the continuity of $Q$ on $\mathcal{S}$ by \Cref{Lemma: eta is Lipschitz continuous}. 
	Consequently, this theorem is obtained with the help of \Cref{prop: stationary point of prox}.
\end{proof}

\section{Sequential convergence and convergence rate of monotone CMPGA} \label{section: Global convergence}

In this section, we investigate the sequential convergence and convergence rate of the entire solution sequence $\{ x^{(t)}:t\in\mathbb{N} \}$ generated by CMPGA with a monotone line-search scheme (CMPGA\_ML), i.e., CMPGA with the parameter $M=0$, under suitable assumptions.

\subsection{Sequential convergence of CMPGA\_ML}
In this subsection, we consider the sequential convergence of CMPGA\_ML. By assuming the KL property of $Q$ defined in \eqref{definition: function Q}, our convergence analysis will show that $\{(x^{(k,0)},y^{(k,0)}):k\in\mathbb{N}\}$ generated by CMPGA\_ML and $Q$ satisfy all the requirements in \Cref{proposition: KL convergence framework} under suitable conditions, and thus yields the sequential convergence of $\{ x^{(t)}:t\in\mathbb{N} \}$ to a critical point of $F$. Specifically, the boundedness of $\{(x^{(k,0)},y^{(k,0)}):k\in\mathbb{N}\}$ generated by CMPGA\_ML is a direct consequence of \Cref{lemma: MPGAc d to 0} \StatementNum{1} and the boundedness of $\mathcal{S}$ defined in \eqref{eq: definition of S}, while Items \StatementNum{1} and \StatementNum{3} of \Cref{proposition: KL convergence framework} are verified for $\{(x^{(k,0)},y^{(k,0)}):k\in\mathbb{N}\}$ and $Q$ in the next proposition. To simplify the notation, we denote $Q(x^{(k,i)},y^{(k,i)})$ by $Q_{(k,i)}$ for $k\in\mathbb{N}$ and $i\in\mathbb{N}_N^0$.

\begin{proposition}\label{proposition: CMPGA sufficient descent}
	Suppose that \Cref{Assumption: X is compact} holds. Let the sequence $\{(x^{(t)},y^{(t)}):t\in\mathbb{N} \}$ be generated by CMPGA\_ML, and $\left\{ \left(x^{(k,i)},y^{(k,i)}\right):(k,i)\in\mathbb{N}\times\mathbb{N}_N^0 \right\}$ be defined by \eqref{eq: z0504 2037}. Then the following two statements hold:
	\begin{enumerate}[label = {\upshape(\roman*)}]
		\item There exist $C_1>0$ and $K_1>0$ such that 
		\begin{equation*}\label{eq: sufficient decrease for MPGAc_ML}
			Q_{(k+1,0)} + C_1 \left( \|x^{(k+1,0)}-x^{(k,0)}\|^2_2 + Q_{(k+1,0)}\|y^{(k+1,0)}-y^{(k,0)}\|^2_2 \right) 
			\leq Q_{(k,0)}
		\end{equation*}
		holds for any $k\geq K_1$;	
		\item $\xi := \lim_{k\to\infty} Q_{(k,0)} $ exists and $Q$ takes the value $\xi$ at any accumulation point of $\{(x^{(k,0)},y^{(k,0)}):k\in\mathbb{N} \}$.
	\end{enumerate}
\end{proposition}	

\begin{proof}
	We first prove Item \StatementNum{1}. The monotone line search scheme adopted by CMPGA\_ML implies that
	\begin{equation}\label{eq: z0317 2149}
		Q_{(k+1,0)} = Q_{(l((k+1)(N+1)+0))} \leq Q_{(l(k(N+1)+i))} = Q_{(k,i)}
	\end{equation}
	holds for any $i\leq N+1$. By dividing by $\eta(x^{(k,i+1)},y^{(k,i+1)})$ on the both sides of \eqref{eq: descent lemma for MPGAc} and combining \eqref{eq: z0317 2149}, we obtain for any $(k,i)\in\mathbb{N}\times\mathbb{N}_N^0$ that
	\begin{equation}\label{eq: z1209 1957}
		Q_{(k,i+1)} + \frac{C_0}{\overline{g}_{\mathcal{X}}} (\|x^{(k,i+1)} - x^{(k,i)}\|^2_2 + Q_{(k,i)}\|y^{(k,i+1)}-y^{(k,i)}\|^2_2)  \leq Q_{(k,i)},
	\end{equation} 
	where $\overline{g}_{\mathcal{X}}\geq \eta(x^{(k,i+1)},y^{(k,i+1)})$ is given by \Cref{Remark: h is global Lips and sup of g} \StatementNum{1}. Summing up \eqref{eq: z1209 1957} over $i\in\mathbb{N}_N^0$, we derive Item \StatementNum{1} with $C_1 = C_0/\overline{g}_{\mathcal{X}}$ in view of \eqref{eq: CMPGA} and \eqref{eq: z0317 2149}.
	
	
	Then, since Item \StatementNum{1} shows that the nonnegative scalar sequence $\{Q_{(k,0)}:k\in\mathbb{N} \}$ decreases monotonously, we conclude Item \StatementNum{2} by the continuity of $Q$ on $\mathcal{S}$ (see \Cref{Lemma: eta is Lipschitz continuous}) and the fact that each accumulation point of $\{(x^{(k,0)},y^{(k,0)}):k\in\mathbb{N} \}$ belongs to $\mathcal{S}$ (see \Cref{theorem: subsequential convergent of MPGAc} \StatementNum{1}). 
\end{proof}

Now it remains to prove that $\{(x^{(k,0)},y^{(k,0)}):k\in\mathbb{N} \}$ generated by CMPGA\_ML satisfies Item \StatementNum{2} of \Cref{proposition: KL convergence framework} with $H:=Q$. To this end, we introduce the following two assumptions and a technical lemma.

\begin{assumption}\label{Assumption: f is L continuous}
	The function $f$ is locally Lipschitz continuous on $\mathcal{X}$.
\end{assumption}

\begin{assumption}\label{Assumption: g* calm}
	The function $g^*$ satisfies calmness condition on $\dom(g^*)$.
\end{assumption}

\begin{lemma}\label{lemma: Q is Lips on S when f is Lips on X}
	Suppose that Assumptions \ref{Assumption: X is compact} and \ref{Assumption: f is L continuous} hold. Let $\{(x^{(t)},y^{(t)}):t\in\mathbb{N} \}$ be generated by CMPGA\_ML. Then there exists $L_Q>0$ such that
	\begin{equation*}
		|Q_{(t_1)}-Q_{(t_2)}|\leq L_Q\Big (\|x^{(t_1)}-x^{(t_2)}\|_2 + \min\{Q_{(t_1)},Q_{(t_2)} \}\|y^{(t_1)}-y^{(t_2)}\|_2  \Big )
	\end{equation*} 
	holds for any $t_1,t_2\in\mathbb{N}$. 
\end{lemma}

\begin{proof}
	Assume $Q_{(t_1)} \geq Q_{(t_2)}$ without loss of generality. Then, we have  
	\begin{align}\notag 
		|Q_{(t_1)}-Q_{(t_2)}| &= \left| \frac{\zeta(x^{(t_1)})}{\eta_{(t_1)}} -\frac{\zeta(x^{(t_2)})}{\eta_{(t_1)}} + \frac{\zeta(x^{(t_2)})}{\eta_{(t_1)}} - \frac{\zeta(x^{(t_2)})}{\eta_{(t_2)}}	\right| \\
		&\leq \frac{|\zeta(x^{(t_1)})-\zeta(x^{(t_2)})|}{\underline{\eta_c}}  +
		\frac{|\eta_{(t_1)}-\eta_{(t_2)}|}{\underline{\eta_c}}Q_{(t_2)},\label{eq: z1209 1730}
	\end{align}
	where $\eta_{(t_i)} = \eta(x^{(t_i)},y^{(t_i)})$, $i=1,2$, and $\underline{\eta_c}>0$ is given by \Cref{theorem: subsequential convergent of MPGAc} \StatementNum{1}. Notice that $\zeta$ is globally Lipschitz continuous on $\mathcal{X}$ under Assumptions \ref{Assumption: X is compact} and \ref{Assumption: f is L continuous}, and $\eta$ is Lipschitz continuous on $\mathcal{S}$ by \Cref{Lemma: eta is Lipschitz continuous}. This lemma is derived from \eqref{eq: z1209 1730} and the fact that $Q_{(t_2)}\leq F(x^{(0)})$.
\end{proof}

With the help of Assumptions \ref{Assumption: X is compact}-\ref{Assumption: g* calm} and \Cref{lemma: Q is Lips on S when f is Lips on X}, we have the following proposition regarding the relative error condition of \Cref{proposition: KL convergence framework}.

\begin{proposition}\label{proposition: CMPGA satisfies relative error}
	Suppose that Assumptions \ref{Assumption: X is compact}-\ref{Assumption: g* calm} hold. Let the sequence $\{(x^{(t)},y^{(t)}):t\in\mathbb{N} \}$ be generated by CMPGA\_ML, and $(x^{(k,i)},y^{(k,i)})$ be defined by \eqref{eq: z0504 2037}. Then, there exist $\widetilde{C}_2>0$, $K_2>0$ and $w^{k+1}\in
	\widehat{\partial}	 Q(x^{(k+1,0)},y^{(k+1,0)})$ such that
	\begin{equation*}
		\|w^{k+1}\|_2 \leq \widetilde{C}_2 \left( \|x^{(k+1,0)}-x^{(k,0)}\|_2+Q_{(k+1,0)}\|y^{(k+1,0)}-y^{(k,0)}\|_2  \right)
	\end{equation*}
	holds for any $k\geq K_2$.
\end{proposition}

\begin{proof}	
	By \Cref{ppsition:2.2} \StatementNum{2}, it suffices to show that some $(w^{k+1}_x,w^{k+1}_y)$ satisfies
	\begin{align}
		&w^{k+1}_x \in  \frac{\widehat{\partial} \zeta(x^{(k+1,0)}) - Q_{(k+1,0)} y^{(k+1,0)}}{\eta_{(k+1,0)}},\label{eq: z0328 1625} \\
		&w^{k+1}_y \in \frac{ Q_{(k+1,0)}(\partial g^*(y^{(k+1,0)})- x^{(k+1,0)})}{\eta_{(k+1,0)}},\label{eq: z0328 1626}
	\end{align}
	and both of $\|w^{k+1}_x\|_2$ and $\|w^{k+1}_y\|_2$ can be bounded by $ \|x^{(k+1,0)}-x^{(k,0)}\|_2+Q_{(k+1,0)}\|y^{(k+1,0)}-y^{(k,0)}\|_2$.
	
	Firstly, it follows from \eqref{eq: Step 2-Y for MPGAc} and \Cref{Lemma: Step 2-Y} \StatementNum{1} that
	$\frac{1}{\alpha_{(k,0)}}(y^{(k,0)}-y^{(k,1)})+x^{(k,0)}\in \partial g^*(y^{(k,1)}) $
	holds for some $\alpha_{(k,0)}\in[\underline{\alpha},\overline{\alpha}]$. By setting
	\begin{equation*}
		w^{k+1}_y = \frac{\frac{Q_{(k+1,0)}}{\alpha_{(k,0)}}(y^{(k,0)}-y^{(k,1)}) + Q_{(k+1,0)}(x^{(k,0)}-x^{(k+1,0)})}{\eta_{(k+1,0)}},
	\end{equation*}
	we derive \eqref{eq: z0328 1626} and $\|w^{k+1}_y\|_2 \leq \frac{Q_{(k+1,0)}}{\underline{\alpha}\underline{\eta_c}} \|y^{(k,0)}-y^{(k+1,0)}\|_2 + \frac{Q_{(k+1,0)}}{\underline{\eta_c}} \|x^{(k,0)}-x^{(k+1,0)}\|_2$ due to $y^{(k,1)} = y^{(k+1,0)}$. 
	This together with the fact that $Q_{(k+1,0)}\leq F(x^{(0)})$ implies some $C_{2,y} >0$ such that
	\begin{equation}\label{eq: z0328 1800}
		\|w^{k+1}_y\|_2 \leq C_{2,y} \left( \|x^{(k,0)}-x^{(k+1,0)}\|_2 + Q_{(k+1,0)}\|y^{(k,0)}-y^{(k+1,0)}\|_2 \right).
	\end{equation}
	Secondly, it is derived from \eqref{eq: Step 2-X for MPGAc} that
	\begin{equation}\label{eq: z1210 2141}
		\frac{1}{\alpha_{(k,i)}}\left(x^{(k,i)}_i-x^{(k,i+1)}_i\right) -  \nabla_i h(x^{(k,i)}) + Q_{(k,i)}y^{(k,i)}_i \in  \widehat{\partial} f_i(x^{(k,i+1)}_i)
	\end{equation}
	holds for some $\alpha_{(k,i)}\in[\underline{\alpha}_{\star},\overline{\alpha}]$. Let  $\Lambda^{(k)},\Xi^{(k)}\in\mathbb{R}^{n\times n}$ be block diagonal matrices whose $i$-th diagonal matrices are $\frac{1}{\alpha_{(k,i)}}I_{n_i}$ and $Q_{(k,i)}I_{n_i}$, respectively, with $I_{n_i}\in\mathbb{R}^{n_i\times n_i}$ being the identity matrix, $i\in\mathbb{N}_N$, and $\wideparen{h}^{(k)}\in\mathbb{R}^n$ be defined by
	$$\wideparen{h}^{(k)} := \left[ \nabla_1 h(x^{(k,1)}),\nabla_2 h(x^{(k,2)}),...,\nabla_N h(x^{(k,N)})\right].$$ 
	Then, by combining \eqref{eq: z1210 2141} over $i\in\mathbb{N}_N$, there holds for CMPGA\_ML that
	\begin{equation}\label{eq: 0625 13}
		\Lambda^{(k)}(x^{(k,0)}-x^{(k+1,0)}) -  \wideparen{h}^{(k)} + \Xi^{(k)} y^{(k+1,0)}\in \widehat{\partial}  f(x^{(k+1,0)}).
	\end{equation}
	Therefore, we derive \eqref{eq: z0328 1625} from \eqref{eq: 0625 13} with
	\begin{multline}\label{eq: z03101603}
		w^{k+1}_x = \Big ( \Lambda^{(k)}(x^{(k,0)}-x^{(k+1,0)}) + \nabla h(x^{(k+1,0)}) - \wideparen{h}^{(k)}\\
		+ \Xi^{(k)} y^{(k+1,0)} - Q_{(k+1,0)} y^{(k+1,0)} 
		\Big ) \eta_{(k+1,0)}^{-1}.
	\end{multline}
	For the term $\nabla h(x^{(k+1,0)}) - \wideparen{h}^{(k)}$ on the right side of \eqref{eq: z03101603}, we have
	\begin{align}
		\|\nabla h(x^{(k+1,0)}) - \wideparen{h}^{(k)}\|_2
		=  \sqrt{\sum_{i=1}^{N} \|\nabla_i h(x^{(k+1,0)}) - \nabla_i h(x^{(k,i)})\|_2^2 } \notag \\
		\leq \sqrt{ \sum_{i=1}^{N} L^2 \|x^{(k+1,0)} - x^{(k,i)}\|_2^2 }
		\leq L \sqrt{N} \|x^{(k+1,0)} - x^{(k,0)}\|_2, \label{eq: z03101612}
	\end{align}
	where the last relation holds due to $\|x^{(k+1,0)} - x^{(k,i)}\|_2\leq \|x^{(k+1,0)} - x^{(k,0)}\|_2$ (see \eqref{eq: CMPGA}).	Besides, according to the term $\Xi^{(k)} y^{(k+1,0)} - Q_{(k+1,0)} y^{(k+1,0)}$ on the right side of \eqref{eq: z03101603}, by exploiting \Cref{lemma: Q is Lips on S when f is Lips on X},	there holds with $I_{n}\in\mathbb{R}^{n\times n}$ being the identity matrix that
	\begin{align}
		&\left\| \Xi^{(k)} - Q_{(k+1,0)} I_{n} \right\|_2 
		= \sup\left\{ |Q_{(k,i)}-Q_{(k+1,0)}| : i\in\mathbb{N}_N \right\} \notag \\
		& \leq L_Q \sup \left\{   \|x^{(k,i)} - x^{(k+1,0)}\|_2 + Q_{(k+1,0)}  \|y^{(k,i)}-y^{(k+1,0)}\|_2 :i\in\mathbb{N}_N  \right\} 
		\notag \\ &
		\leq L_Q  \|x^{(k,0)} - x^{(k+1,0)}\|_2, \label{eq: z03101816}
	\end{align}
	where the second relation follows from \Cref{lemma: Q is Lips on S when f is Lips on X}, and the third relation follows from \eqref{eq: CMPGA}.  Hence, the equation \eqref{eq: z03101603} leads to some $C_{2,x}>0$ so that
	\begin{equation}
		\|w^{k+1}_x\|_2 
		\leq C_{2,x}  \|x^{(k,0)}-x^{(k+1,0)}\|_2  ,  \label{eq: z0328 1759}
	\end{equation}
	owing to $\|\Lambda^{(k)}\|_2\leq \underline{\alpha}_{\star}^{-1}$, \eqref{eq: z03101612}, \eqref{eq: z03101816}, the boundedness of $\{y^{(k,0)}:k\in\mathbb{N}\}$ and \Cref{theorem: subsequential convergent of MPGAc} \StatementNum{1}. Finally, this lemma follows from \eqref{eq: z0328 1800} and \eqref{eq: z0328 1759}.
\end{proof}

In view of \Cref{lemma: MPGAc d to 0} \StatementNum{3} and \Cref{proposition: CMPGA satisfies relative error}, we get the following corollary.

\begin{corollary}\label{corollary: CMPGA subdiff converge to 0}
	Suppose that Assumptions \ref{Assumption: X is compact}-\ref{Assumption: g* calm} hold. Let the sequence $\{(x^{(t)},y^{(t)}):t\in\mathbb{N} \}$ be generated by CMPGA\_ML, and $(x^{(k,i)},y^{(k,i)})$ be defined by \eqref{eq: z0504 2037}. Then, 
	\begin{equation}\label{eq: CMPGA subdiff converge to 0}
		\lim_{k\to\infty} \dist\left( 0,\widehat{\partial} Q(x^{(k,0)},y^{(k,0)}) \right) = 0.
	\end{equation}
\end{corollary}

We remark that, by following the same arguments, one can show that results in \Cref{lemma: Q is Lips on S when f is Lips on X} and \Cref{proposition: CMPGA satisfies relative error} also hold when $\{(x^{(t)},y^{(t)}):t\in\mathbb{N} \}$ is generated by CMPGA equipped with a nonmonotone line-search scheme, i.e., $M>0$. Consequently, the convergence result \eqref{eq: CMPGA subdiff converge to 0} obtained in \Cref{corollary: CMPGA subdiff converge to 0} holds for CMPGA, regardless of the choice of $M$.

With the help of \Cref{proposition: CMPGA sufficient descent} and \Cref{proposition: CMPGA satisfies relative error}, we are now ready to establish the sequential convergence of CMPGA\_ML.


\begin{theorem}\label{theorem: CMPGA is global convergence}
	Suppose that Assumptions 2-4 hold.  If $Q$ satisfies the KL property at any $(x,y)\in\mathcal{S}$, then the sequence $\{ x^{(t)}:t\in\mathbb{N} \}$ generated by CMPGA\_ML converges to a critical point $x^{\star}$ of $F$.
\end{theorem}	

\begin{proof}
	\Cref{proposition: CMPGA satisfies relative error} yields some $K_2>0$ and some $w^{k+1}\in
	\widehat{\partial}	 Q(x^{(k+1,0)},y^{(k+1,0)})$ such that
	\begin{equation*}
		\|w^{k+1}\|_2 \leq C_2 \left( \|x^{(k+1,0)}-x^{(k,0)}\|_2+\sqrt{Q_{(k+1,0)}}\|y^{(k+1,0)}-y^{(k,0)}\|_2  \right)
	\end{equation*}
	holds for any $k\geq K_2$ with $C_2 := \widetilde{C_2}\max\{1,\sqrt{F(x^{(0)})} \}$ owing to $Q_{(k+1,0)}\leq F(x^{(0)})$. This indicates that the condition \StatementNum{2} in \Cref{proposition: KL convergence framework} holds for the sequence $\{ (x^{(k,0)},y^{(k,0)}):k\in\mathbb{N} \} $ generated by CMPGA\_ML. Then, invoking \Cref{proposition: CMPGA sufficient descent}, we derive from \Cref{proposition: KL convergence framework} that
	\begin{equation}\label{eq: z1211 1627}
		\sum_{k=0}^{\infty} \|x^{(k+1,0)} - x^{(k,0)}\|_2 < \infty.
	\end{equation}
	Notice that, according to CMPGA, the relation \eqref{eq: CMPGA} yields 
	\begin{equation}\label{eq: z0317 2239}
		\sum_{i=1}^N \|x^{(k,i+1)} - x^{(k,i)}\|_2 = \sum_{i=1}^N \|x^{(k+1,0)}_i - x^{(k,0)}_i\|_2 \leq \sqrt{N} \|x^{(k+1,0)} - x^{(k,0)}\|_2.
	\end{equation}
	In view of \eqref{eq: z1211 1627} and \eqref{eq: z0317 2239}, we derive 
	\begin{equation*}
		\sum_{t=0}^{\infty} \|x^{(t+1)} - x^{(t)}\|_2
		= \sum_{k=0}^{\infty} \sum_{i=1}^N \|x^{(k,i+1)} - x^{(k,i)}\|_2 
		< +\infty.
	\end{equation*}
	Therefore, the sequence $\{ x^{(t)}:t\in\mathbb{N} \}$ converges to some  $x^{\star}\in\mathbb{R}^n$, and \Cref{theorem: subsequential convergent of MPGAc} \StatementNum{2} demonstrates that the $x^{\star}$ belongs to $\dom(F)$ and is a critical point of $F$. This completes the proof.
\end{proof}	

Finally, we shall show that Assumptions \ref{Assumption: f is L continuous}-\ref{Assumption: g* calm} and the KL condition in \Cref{theorem: CMPGA is global convergence} hold for the sparse signal recovery problems \eqref{problem:L1dL2} and \eqref{problem:L1dSK}, which together with \Cref{Assumption: X is compact} lead to the sequential convergence of CMPGA\_ML for solving these two problems. Recall that $f$ in both problems can be viewed as the sum of $\ell_1$ norm and the indicator function on $\{x\in\mathbb{R}^n:\underline{x} \leq  x \leq \overline{x} \} $, while $g^*$ is indicator functions on the $\ell_2$ unit ball and $\{x\in\mathbb{R}^n: \|x\|_{\infty}\leq 1, \|x\|_1 \leq K \} $ respectively. Hence, Assumptions \ref{Assumption: f is L continuous} and \ref{Assumption: g* calm} are automatically fulfilled for problems \eqref{problem:L1dL2} and \eqref{problem:L1dSK}. In the literature, the KL property of a concrete function is often verified via showing that it is a proper, closed and semi-algebraic function. Since $Q$ may be not lower semicontinuous on $\mathbb{R}^n\times \mathbb{R}^n$, we could not deduce its KL property even if it is a proper semi-algebraic function. To remedy this issue, we introduce a potential function $Q_{\epsilon}:\mathbb{R}^n\times \mathbb{R}^n \to (-\infty,+\infty] $ for $\epsilon>0$, which is defined as the sum of $Q$ and the indicator function on the set $\{(x,y)\in\mathbb{R}^n\times \mathbb{R}^n: \eta(x,y)\geq \epsilon \} $. Clearly, for any $\epsilon>0$, $Q_{\epsilon}$ is closed on $\mathbb{R}^n\times \mathbb{R}^n$. Furthermore, we show the connection between the KL property of $Q$ and $Q_{\epsilon}$ in the next proposition.

\begin{proposition}\label{proposition: Q+eta is KL -> Q is KL}
	Suppose that $g^*$ is continuous on $\dom(g^*)$. If for any $\epsilon>0$, $Q_{\epsilon}$ is a KL function, then $Q$ is also a KL function.
\end{proposition}
\begin{proof}
	Let $(x,y)\in\dom(\partial Q)$ and $\epsilon:=\frac{1}{2}\eta(x,y)$. Then one can check that $Q = Q_{\epsilon} $ holds on $\mathcal{B}((x,y),\delta)$ for some $\delta>0$ due to the continuity of $g^*$ on $\dom(g^*)$. Therefore, $Q$ satisfies the KL property at $(x,y)$ since $Q_{\epsilon}$ is a KL function.
\end{proof}

By the definition of a semi-algebraic function, one can easily check that for any $\epsilon>0$, the respective $Q_{\epsilon}$ for problems \eqref{problem:L1dL2} and \eqref{problem:L1dSK} is a proper closed semi-algebraic function, which together with \cite[Theorem 4.1]{Attouch-bolt-redont-soubeyran:2010} indicates that $Q_{\epsilon}$ is a KL function for any $\epsilon>0$. Hence, with the help of \Cref{proposition: Q+eta is KL -> Q is KL}, we conclude that the respective $Q$ for these two problems is a KL function. Using the above display and the discussions below \Cref{Assumption: X is compact}, we finally obtain the following theorems concerning the sequential convergence of CMPGA\_ML for solving problems \eqref{problem:L1dL2} and \eqref{problem:L1dSK}, respectively.

\begin{theorem}
	The sequence $\{ x^{(t)}:t\in\mathbb{N} \}$ generated by CMPGA\_ML for problem \eqref{problem:L1dL2} converges to a critical point of $F$, if $F(x^{(0)})< 1+\frac{\lambda}{2}\|b\|_2^2$.
\end{theorem}

\begin{theorem}\label{theorem: CMPGA converges to critical}
	The sequence $\{ x^{(t)}:t\in\mathbb{N} \}$ generated by CMPGA\_ML for problem \eqref{problem:L1dSK} globally converges to a critical point of $F$. 
\end{theorem}

\subsection{Convergence rate of CMPGA\_ML}
This subsection is devoted into the convergence rate analysis of CMPGA\_ML. We first deduce generally the asymptotic convergence rate of CMPGA\_ML based on the KL exponent of $Q$ in the next theorem. The proof follows a similar line of arguments to other convergence rate analysis based on the KL property (see \cite{Attouch-Bolte:MP2009,Wen-Chen-Pong:2018COA,Xu-Yin:SIAMIS:13}, for example) and thus are omitted here for brevity.

\begin{theorem}\label{theorem: convergence rate for KL exponents}
	Suppose that Assumptions 2-4 hold. Let $\{ x^{(t)}:t\in\mathbb{N} \}$ be the sequence generated by CMPGA\_ML. If $\{ x^{(t)}:t\in\mathbb{N} \}$ converges to some $x^{\star}\in\mathbb{R}^n$, and $Q$ satisfies the KL property at each $(x^{\star},y)\in\mathcal{S}$ with an exponent $\theta\in[0,1)$, then the following statements hold:
	\begin{enumerate}[label = {\upshape(\roman*)}, leftmargin=0.3cm, itemindent=0.3cm]
		\item If $\theta = 0$, $\{ x^{(t)}:t\in\mathbb{N} \}$ converges to $x^{\star}$ finitely;
		\item If $\theta \in (0,1/2]$, $\|x^{(t)}-x^{\star}\|_2\leq c_1 \tau^t$, $\forall t\geq T_1$ for some $T_1>0$, $c_1>0$, $\tau\in(0,1)$;
		\item If $\theta \in (1/2,1)$, $\|x^{(t)}-x^{\star}\|_2\leq c_2 t^{-(1-\theta)/(2\theta -1)}$, $\forall t\geq T_2$, for some $T_2>0$, $c_2>0$.
	\end{enumerate}	
\end{theorem}

While it is routine to obtain the results in \Cref{theorem: convergence rate for KL exponents}, it is more challenging to estimate the KL exponent of $Q$ for a concrete model of \eqref{problem:root}. In what follows, we shall show that the KL exponent of $Q$ is $\frac{1}{2}$ under further assumptions on $f$, $h$ and $g$ in problem \eqref{problem:root}. These assumptions, in particular, hold for the sparse signal recovery problem \eqref{problem:L1dSK}. The following theorem tells that we can estimate the KL exponents of $Q$ via a non-fractional auxiliary function, whose proof is inspired by \cite[Theorem 4.2]{Zeng-Yu-Pong:2021SIAM-OPT}. We point out that the result of \cite[Theorem 4.2]{Zeng-Yu-Pong:2021SIAM-OPT} can not be directly applied to \Cref{proposition: KL exponents psi_alpha to Q}, since $g^*$ is not necessarily differentiable.

\begin{theorem}\label{proposition: KL exponents psi_alpha to Q}
	Let $(\overline{x},\overline{y})\in \dom(\partial Q)$, $\alpha := Q(\overline{x},\overline{y})$ and $g^*$ satisfy the calmness condition on $\dom(g^*)$. Let $\psi_{\alpha}:\mathbb{R}^n\times\mathbb{R}^n\to(-\infty,+\infty]$ be defined at $(x,y)\in\mathbb{R}^n\times \mathbb{R}^n$ as
	\begin{equation*}
		\psi_{\alpha}(x,y) := \begin{cases}
			\zeta(x) + \alpha(g^*(y)-\innerP{x}{y}),  &\text{if $(x,y)\in\dom(\zeta)\times\dom(g^*)$},\\
			+\infty, &\text{else.}
		\end{cases}
	\end{equation*}
	If $\psi_{\alpha}$ satisfies the KL property with the exponent $\theta\in [0,1)$ at $(\overline{x},\overline{y})$, then $Q$ satisfies the KL property with the same exponent $\theta\in [0,1)$ at $(\overline{x},\overline{y})$.
\end{theorem}

\begin{proof}
	Since $\psi_{\alpha}$ has the KL exponent $\theta$ at $(\overline{x},\overline{y})$ and $\psi_{\alpha}(\overline{x},\overline{y}) = 0$, there exist $\epsilon,\delta_1,c>0$ such that
	\begin{equation}\label{eq: z0904 1620}
		\dist(0,\partial\psi_{\alpha}(x,y))\geq c(\psi_{\alpha}(x,y))^{\theta}
	\end{equation}
	holds for any $(x,y)\in\mathcal{B}^{\epsilon}_{\psi_{\alpha}}((\overline{x},\overline{y}),\delta_1)$. It is worthy noting that $\mathcal{B}^{\epsilon}_{\psi_{\alpha}}((\overline{x},\overline{y}),\delta_1) \subseteq \dom(\zeta)\times\dom(g^*) $. Let $\eta(x,y):=\innerP{x}{y}-g^*(y)$ and $G:=\|\overline{y}\|_2 + 2\eta(\overline{x},\overline{y}) + \delta_1$. Owing to the continuity of $\eta$ and $\frac{\psi_{\alpha}^{1-\theta}}{\eta}$ at $(\overline{x},\overline{y})$ and the fact that  $\frac{\psi_{\alpha}^{1-\theta}(\overline{x},\overline{y})}{\eta(\overline{x},\overline{y})} = 0$, there exists some $\delta_2\leq \delta_1$, such that, for any $(x,y)\in\mathcal{B}^{\epsilon}_{\psi_{\alpha}}((\overline{x},\overline{y}),\delta_2)$, it holds
	\begin{align}\label{eq: z0904 1521}
		&\frac{\eta(\overline{x},\overline{y})}{2}\leq \eta(x,y)\leq G,\\
		&\frac{\psi_{\alpha}^{1-\theta}(x,y)}{\eta(x,y)}\leq \frac{c}{2G}.\label{eq: z0904 1522}
	\end{align}
	Set $\epsilon' = \epsilon/G$. Then  $0<\psi_{\alpha}(x,y)<\epsilon'G=\epsilon$ holds for any $(x,y)\in\mathcal{B}^{\epsilon'}_{Q}((\overline{x},\overline{y}),\delta_2)$ upon the fact that 
	\begin{equation}\label{eq: z0904 1622}
		\psi_{\alpha}(x,y) = (Q(x,y)-\alpha)\eta(x,y).
	\end{equation}
	Therefore, we deduce that \eqref{eq: z0904 1521} and \eqref{eq: z0904 1522} hold for any $(x,y)\in\mathcal{B}^{\epsilon'}_{Q}((\overline{x},\overline{y}),\delta_2)$.
	
	Let $(x,y)\in\mathcal{B}^{\epsilon'}_{Q}((\overline{x},\overline{y}),\delta_2)\cap\dom(\partial Q)$. Then we have $\zeta(x)>0$ and $\eta(x,y)>0$.
	By \Cref{Corollary: limiting subdiff of frac} \StatementNum{2}, there holds that
	\begin{align*}
		&\dist(0,\partial Q(x,y))
		= \frac{1}{\eta(x,y)}\dist \Big ( 0,(\partial\zeta(x)-Q(x,y)y) \times Q(x,y)(\partial g^*(y)-x) \Big )\\
		&= \frac{1}{\eta(x,y)}\inf\Big\{ \sqrt{\|u-Q(x,y)y\|^2_2+\|Q(x,y)(v-x)\|^2_2}:(u,v)\in\partial\zeta(x)\times\partial g^*(y) \Big\}\\
		&\overset{(\uppercase\expandafter{\romannumeral 1})}{\geq}
		\frac{1}{\sqrt{2}G}\Big ( \inf\{\|u-Q(x,y)y\|_2:u\in\partial\zeta(x)  \} +\inf\{\|Q(\overline{x},\overline{y})(v-x)\|_2:v\in\partial g^*(y) \} \Big )\\
		&\overset{(\uppercase\expandafter{\romannumeral 2})}{\geq}	
		\frac{1}{\sqrt{2}G}\Big ( \inf\{\|u-\alpha y\|_2:u\in\partial\zeta(x)  \} +\inf\{\|\alpha (v-x)\|_2:v\in\partial g^*(y) \} \Big )\\
		&\qquad -|Q(x,y)-\alpha | \frac{\|y\|_2}{\sqrt{2}G}
		\overset{(\uppercase\expandafter{\romannumeral 3})}{\geq}
		\frac{1}{\sqrt{2}G} \dist(0,\partial \psi_{\alpha}(x,y)) -\frac{1}{\sqrt{2}}|Q(x,y)-\alpha |\\
		&\overset{(\uppercase\expandafter{\romannumeral 4})}{\geq}
		\frac{1}{\sqrt{2}}\left( \frac{c}{G}(\psi_{\alpha}(x,y))^{\theta}-\frac{\psi_{\alpha}(x,y)}{\eta(x,y)} \right) 
		= \frac{1}{\sqrt{2}}\left( \frac{c}{G}-\frac{\psi_{\alpha}^{1-\theta}(x,y)}{\eta(x,y)} \right) (\psi_{\alpha}(x,y))^{\theta}\\
		&\overset{(\uppercase\expandafter{\romannumeral 5})}{\geq} \frac{c}{2\sqrt{2}G}\psi_{\alpha}^{\theta}(x,y) 
		\overset{(\uppercase\expandafter{\romannumeral 6})}{=}\frac{c}{2\sqrt{2}G} (\eta(x,y))^{\theta} (Q(x,y)-\alpha)^{\theta}\\
		&\overset{(\uppercase\expandafter{\romannumeral 7})}{\geq} \frac{c}{2\sqrt{2}G}\left( \frac{\eta(\overline{x},\overline{y})}{2} \right)^{\theta} (Q(x,y)-\alpha)^{\theta},
	\end{align*}
	where $(\uppercase\expandafter{\romannumeral 1})$ follows from \eqref{eq: z0904 1521}, $Q(\overline{x},\overline{y}) < Q(x,y)$, and the fact that $\sqrt{\alpha^2+\beta^2}\geq (\alpha+\beta)/\sqrt{2}$ holds for any $\alpha,\beta\geq 0$; $(\uppercase\expandafter{\romannumeral 2})$ comes from the triangle inequality; $(\uppercase\expandafter{\romannumeral 3})$ is derived from
	the relation $(\partial\zeta(x)-\alpha y)\times \alpha(\partial g^*(y)-x) \subseteq \partial \psi_{\alpha}(x,y) $ and the fact that $\|y\|_2 \leq \delta_2+\|\overline{y}\|_2 \leq G$; $(\uppercase\expandafter{\romannumeral 4})$ holds due to \eqref{eq: z0904 1620} and \eqref{eq: z0904 1622}; $(\uppercase\expandafter{\romannumeral 5})-(\uppercase\expandafter{\romannumeral 7})$ respectively come from \eqref{eq: z0904 1522}, \eqref{eq: z0904 1622} and \eqref{eq: z0904 1521}.	
\end{proof}

It is proved by \cite[Corollary 5.2]{Li-Pong:2018Foundations_of_computational_mathematics} that the proper closed function $f:\mathbb{R}^n \to (-\infty,+\infty]$ is a KL function with the exponent $1/2$, if $f$ can be written as 
\begin{equation}\label{eq: z03071436}
	f(x) = \min_{1\leq i\leq T}\left\{\Psi_i(x) + \frac{1}{2}x^{\top}P_ix + q_i^{\top}x+s_i \right \},
\end{equation}
where $\Psi_i:\mathbb{R}^n\to(-\infty,+\infty]$ are proper closed polyhedral functions, $P_i\in\mathbb{R}^{n\times n}$ are symmetric, $q_i\in\mathbb{R}^n$ and $s_i\in\mathbb{R}$ for $i\in\mathbb{N}_T$ and some $T\in\mathbb{N}$.
Recall that a function $\varphi:\mathbb{R}^n \to (-\infty,+\infty]$ is said to be a polyhedral function if its epigraph $\epi(\varphi):=\{(x,\alpha)\in\mathbb{R}^{n+1}:\varphi(x)\leq \alpha  \} $ is a polyhedral set, and a set $\mathcal{S}\subseteq \mathbb{R}^n$ is said to be polyhedral if there exists a finite set $\mathcal{A}\subseteq \mathbb{R}^{n+1}$ such that $\mathcal{S} = \bigcap_{a\in\mathcal{A}} \{x\in\mathbb{R}^n:\innerP{a}{(x,1)}\leq 0  \} $. For the case when $f$ has the form \eqref{eq: z03071436}, the next proposition demonstrates that the KL exponents of $Q$ and $F$ are $1/2$ under some further requirements. 


\begin{proposition}\label{proposition: KL=1/2 when g is polyhedral}
	Suppose that $f$ can be formulated as \eqref{eq: z03071436}, $h$ is a quadratic function, and one of the following assumption on $g$ holds:
	\begin{enumerate}[label = {\upshape(\roman*)}, leftmargin=0.5cm, itemindent=0cm]
		\item $g$ is a polyhedral function;
		\item $g$ is a positive semi-definite quadratic function, that is, $g(x) = \frac{1}{2}x^{\top}Px+q^{\top}x+s$ for some $q\in\mathbb{R}^n$, $s\in\mathbb{R}$, and symmetric positive semi-definite matrix $P\in\mathbb{R}^{n\times n}$.
	\end{enumerate}
	Then the function $F$ given by \eqref{definition: function F} and the function $Q$ given by \eqref{definition: function Q} are both KL functions with the exponent $\theta = 1/2$.
\end{proposition}

\begin{proof}
	For Item \StatementNum{1}, it is revealed by \cite[Theorem 19.2]{Rockafellar:70} that $\dom(g^*)$ is polyhedral, and $g^* = \widehat{\Psi}_* + \iota_{\dom(g^*)}$ holds for some real-valued polyhedral function $\widehat{\Psi}_*$. As to Item \StatementNum{2}, we have $g^*(y) = \frac{1}{2}(y-q)^{\top}P^{\dagger}(y-q)-s + \iota_{\mathrm{Range}(P)+\{q \}}(y),$
	where $P^{\dagger}$ denotes the Moore-Penrose inverse of $P$. Generally speaking, either for Item \StatementNum{1} or Item \StatementNum{2}, the conjugate of $g$ takes the form
	$$g^*(y) = \widehat{\Psi}_*(y) + \iota_{\mathcal{A}}(y) + \frac{1}{2}\, y^{\top}P_*y + q^{\top}_* y + s_*, $$
	where $\widehat{\Psi}_*:\mathbb{R}^n\to\mathbb{R}$ is a real-valued polyhedral function, $\mathcal{A}\subseteq\mathbb{R}^n$ is a polyhedral set, $P_*\in\mathbb{R}^{n\times n}$ is symmetric, $q_*\in\mathbb{R}^{n}$ and $s_*\in\mathbb{R}$. Then, the auxiliary function introduced in \Cref{proposition: KL exponents psi_alpha to Q} can be rewritten as
	\begin{multline*}
		\psi_{\alpha}(x,y) =\\ \min_{1\leq i\leq T}
		\left\{ 
		\frac{1}{2}\Vxy^{\top} \begin{bmatrix} P_i & -\alpha I_n \\ -\alpha I_n & \alpha P_* \end{bmatrix} \Vxy 
		+ \begin{bmatrix} q_i \\ \alpha  q_*	\end{bmatrix}^{\top} \Vxy + \widetilde{s}_i +h(x) + \widetilde{\Psi}_i(x,y)	 \right\},
	\end{multline*}
	where $\alpha := Q(\overline{x},\overline{y})$ for the given $(\overline{x},\overline{y})\in\dom(\partial Q)$, $\widetilde{s}_i := s_i + \alpha s_*$, and $\widetilde{\Psi}_i(x,y):= \Psi_i(x) + \alpha\widehat{\Psi}_*(y) + \iota_{\mathcal{A}}(y) $ is a polyhedral function for each $i\in\mathbb{N}_T$. 
	Also by exploiting  \cite[Corollary 5.2]{Li-Pong:2018Foundations_of_computational_mathematics}, we derive that, for an arbitrary $\alpha\geq 0$, the $\psi_{\alpha}$ is a KL function with the exponent $\theta=1/2$. Note that $g^*$ satisfies the calmness condition on $\dom(g^*)$, since $g^*(y)$ takes the value by the real-valued convex function $\widehat{\Psi}_*(y) + y^{\top}P_*y + q^{\top}_* y + s_* $ for any $y\in\dom(g^*)$. Therefore, we deduce that $Q$ is a KL function with the exponent $\theta=1/2$ with the help of \Cref{proposition: KL exponents psi_alpha to Q}, and $F$ is a KL function with the same exponent $\theta=1/2$ with the help of \Cref{theorem: KL exponents Q to F}. 
\end{proof}

We notice that the vector $K$-norm is a polyhedral function, since
\begin{multline*}
	\epi(\|\cdot\|_{(K)}) = \{(x,\alpha)\in\mathbb{R}^{n+1}: \|x\|_{(K)}\leq \alpha \} \\
	= \bigcap_{a\in\mathcal{A}}\{(x,\alpha)\in\mathbb{R}^{n+1}:a^{\top}x\leq \alpha \}
	= \bigcap_{a\in\mathcal{A}}\left\{(x,\alpha)\in\mathbb{R}^{n+1}:\innerP{
		\begin{bmatrix}
			a\\-1
	\end{bmatrix}}
	{\begin{bmatrix}
			x\\ \alpha
	\end{bmatrix}} \leq 0 \right\},
\end{multline*}
where $\mathcal{A}=\{a\in\{-1,0,1 \}^n:\|a\|_0 = K \}$ is a finite set. Thanks to \Cref{proposition: KL=1/2 when g is polyhedral}, we conclude that the extended objective $F$ and $Q$ for problem \eqref{problem:L1dSK} are both KL functions with exponent $\frac{1}{2}$. By \Cref{theorem: CMPGA converges to critical} and \Cref{theorem: convergence rate for KL exponents}, it leads to the following convergence rate theorem concerning CMPGA\_ML for solving problem \eqref{problem:L1dSK}.

\begin{theorem}
	The solution sequence $\{ x^{(t)}:t\in\mathbb{N} \}$ generated by CMPGA\_ML for problem \eqref{problem:L1dSK} converges R-linearly to a critical point of $F$.
\end{theorem}

\section{Numerical experiments} \label{section: Numerical experiments}

In this section, we conduct some preliminary numerical experiments to test the performance of our proposed algorithms, namely, CMPGA and RMPGA. All the experiments are conducted in MATLAB R2022a on a desktop with an Intel Core i5-9500 CPU (3.00 GHz) and 16GB of RAM.

In our numerical tests, we focus on the two ratio regularized sparse signal recovery problems \eqref{problem:L1dL2} and \eqref{problem:L1dSK} with highly coherent matrices $A$, on which the standard $\ell_1$ sparse signal recovery model usually fails. Following the works of \cite{Rahimi-Wang-Dong-Lou:2019SIAM-SC,Wang-Yan-Rahimi-Lou:2020IEEE}, the matrix $A$ is generated by oversampled discrete cosine transform (DCT), i.e., $A=[a_1,a_2,...,a_n]\in\mathbb{R}^{m\times n}$ with
\begin{equation}\label{def: DCT-Matrix A}
	a_j = \frac{1}{\sqrt{m}}\cos\left( \frac{2\pi \omega j}{D} \right), ~j\in\mathbb{N}_n,
\end{equation}
where $\omega\in\mathbb{R}^m$ is a random vector following the uniform distribution in $[0,1]^m$ and $D>0$ is a parameter measuring how coherent the matrix is. The parameters of the proposed algorithms are set as follows. We set $M=2$, $\sigma = 10^{-6}$, $\gamma = 0.5$, $\widetilde{\alpha}_Y \equiv 1000$ and $\overline{\alpha}=10^8$ throughout the experiments. Motivated from Barzilai-Borwein spectral method \cite{Barzilai-Borwein:IJNA1998}, $\widetilde{\alpha}_{(t)}$ is updated by the following formula:
\begin{equation}\label{eq: BB step-size}
	\widetilde{\alpha}_{(t)} = 
	\begin{cases}
		\max\left\{ \underline{\alpha}, \min \left\{\overline{\alpha},\frac{\|\Delta x^{(t)}\|^2_2}{|\innerP{\Delta x^{(t)}}{\Delta h^{(t)}} |}  \right\} \right\}, & \text{if  } \left|\innerP{\Delta x^{(t)}}{\Delta h^{(t)}} \right|\geq 10^{-12},\\
		\widetilde{\alpha}_{(t-1)}, &\text{else,}
	\end{cases}
\end{equation}
where $\Delta x^{(t)} := x^{(t)}-x^{(t-1)}$, $\Delta h^{(t)} := \nabla h(x^{(t)})-\nabla h(x^{(t-1)}) $. Moreover, we set $\widetilde{\alpha}_{(0)} = 1$ and $\underline{\alpha} = 10^{-8}$ in \eqref{eq: BB step-size} for solving problem \eqref{problem:L1dL2}, while both $\widetilde{\alpha}_{(0)}$ and $\underline{\alpha}$ are chosen as $1.99/(\lambda\|A\|_2^2)$ for problem \eqref{problem:L1dSK} due to the global Lipschitz differentiability of $h$ involved. Besides, the index $i\in\mathbb{N}_N^0$ is picked according to the uniform probability distribution in Step 2 of RMPGA, i.e., $p_j^{(t)}\equiv 1/(N+1)$ holds for any $j\in\mathbb{N}_N^0$ and $t\in\mathbb{N}$. 

\subsection{$L_1/S_K$ sparse signal recovery problem \eqref{problem:L1dSK}}\label{section: Exp-L1dSK}

We first generate random instances of problem \eqref{problem:L1dSK}, following similar experimental setting to those of \cite[Section 5]{Li-Shen-Zhang-Zhou:2022ACHA}. Given $m\in\mathbb{N}$, $n\in\mathbb{N}$ and $D>0$, the matrix $A$ is simulated as \eqref{def: DCT-Matrix A}. The ground truth $x^{\dagger}\in\mathbb{R}^n$ is an $r$-sparse signal that has a minimum separation of $2D$ in its support. Moreover, the nonzero entries of $x^{\dagger}$ are set as $1$ or $-1$ according to two-point distribution. We compute $b\in\mathbb{R}^m$ by $b=Ax^{\dagger}$ and choose $\underline{x} = -2\times \One_n$ and $\overline{x} = 2\times \One_n$, where $\One_n$ denotes the $n$-dimensional vector with all entries being 1. Now we get an instance of problem \eqref{problem:L1dSK} of which $x^{\dagger}$ is a critical point as shown in \cite[Section 5]{Li-Shen-Zhang-Zhou:2022ACHA}. 

We compare the proposed algorithms with PGSA\_NL \cite{NaZhang-QiaLi:2022SIAM-OPT} and PGSA\_BE \cite{Li-Shen-Zhang-Zhou:2022ACHA} for solving problem \eqref{problem:L1dSK}. The parameters of these two algorithms are set as suggested in \cite[Section 5]{Li-Shen-Zhang-Zhou:2022ACHA}. Note that the proposed algorithms involve the proximity operator with respect to the conjugate function of $g:=\|x\|_{(K)}$, which can be evaluated using the fact $\prox_{\alpha g^*}(z) = z-\alpha\prox_{g/\alpha}(\frac{z}{\alpha})$, for $\alpha>0$ and $z\in\mathbb{R}^n$. In our implementations, we adopt Algorithm 3 in \cite{Bogdan-Berg-Sabatti-Su-Candes:2015SLOPE} to compute $\prox_{\|\cdot\|_{(K)}}$, of which the computational complexity is $O(n\log n + n)$. Throughout the experiments, we fix $m=640$, $n=5400$, $r=100$, $\lambda = 200$, and vary the coherent parameter $D$ from $\{1,2,...,9,10 \}$. For each $D$, we generate 50 instances randomly as described above. For each instance, the same initial point $x^{(0)} = x^{\dagger}+0.2\widetilde{e}$ is selected for all the algorithms, where the entries of $\widetilde{e}\in\mathbb{R}^n$ are drawn randomly from the uniform distribution on $[-1,1]^n$. Moreover, all the algorithms are terminated when they reach the precision  
\begin{equation}\label{eq: z0730 1610}
	\frac{\|x^{(k,0)}-x^{\dagger}\|_2}{\|x^{\dagger}\|_2} < 10^{-3}.
\end{equation}

\Cref{fig:blocksvstime} shows a study of the performance of the proposed algorithms with different number of blocks $N$. Figures \ref{fig:blocksvstime}(a) and \ref{fig:blocksvstime}(b) respectively plot the averaged number of epochs ($n/(N+1)$ iterations) and computational time over 50 instances with $D=10$, when CMPGA and RMPGA use various $N$ ranging from 1 to 40. As we can see from \Cref{fig:blocksvstime}(a), the number of epochs required decreases when the number of blocks is increased, which indicates that using larger the number of blocks is more efficient in terms of the overall computation work (e.g., measured in total flops). However, this does not mean shorter computational time. One can observe from \Cref{fig:blocksvstime}(b) that using an appropriate number of blocks may take the least amount of computational time, which is not surprising since less number of blocks may  benefit better from modern multi-core computers for parallel computing and thus help to reduce the actual computational time.

\Cref{Table: L1dSk-Epoch} summarizes the results averaged over 50 instances with $D$ for PGSA\_NL, PGSA\_BE, CMPGA with $N=1$ and $8$, and RMPGA with $N=8$. For each $D$, the averaged number of epochs and computational time (in seconds) for these algorithms are presented in Tables \ref{Table: L1dSk-Epoch} and \ref{Table: L1dSk-time}, respectively. We can find that CMPGA with $N=1$ performs comparably to PGSA\_NL and PGSA\_BE, while CMPGA and RMPGA with $N=8$ outperform the other algorithms.


\begin{table}[htbp]
	\centering
	\caption{Results for averaged epochs for solving problem \eqref{problem:L1dSK}.}
	\begin{tabular}{llrrrrr}
		\toprule
		\raisebox{-3pt}[0pt]{$D$} & ~~~~~ & \raisebox{-3pt}[0pt]{PGSA\_NL} & \raisebox{-3pt}[0pt]{PGSA\_BE} & $\mathop{\mathrm{CMPGA}}\limits_{\tiny(N=1)}$ & $\mathop{\mathrm{CMPGA}}\limits_{(N=8)}$ & $\mathop{\mathrm{RMPGA}}\limits_{(N=8)}$ \\
		\midrule
		1  &   & 170   & 217   & 163   & 65    & 81  \\
		2  &   & 171   & 217   & 163   & 64    & 81  \\
		3  &   & 187   & 226   & 178   & 64    & 86  \\
		4  &   & 199   & 236   & 191   & 71    & 94  \\
		5  &   & 212   & 242   & 203   & 82    & 107  \\
		6  &   & 224   & 251   & 215   & 93    & 120  \\
		7  &   & 236   & 262   & 225   & 105   & 134  \\
		8  &   & 247   & 272   & 235   & 121   & 145  \\
		9  &   & 257   & 285   & 244   & 133   & 155  \\
		10 &   & 267   & 295   & 253   & 148   & 165  \\
		\bottomrule
	\end{tabular}%
	\label{Table: L1dSk-Epoch}%
\end{table}%

\begin{table}[htbp]
	\caption{Results for averaged computational time for solving problem \eqref{problem:L1dSK}.}
	\begin{tabular}{llrrrrr}
		\toprule
		\raisebox{-3pt}[0pt]{$D$} & ~~~~~ & \raisebox{-3pt}[0pt]{PGSA\_NL} & \raisebox{-3pt}[0pt]{PGSA\_BE} & $\mathop{\mathrm{CMPGA}}\limits_{\tiny(N=1)}$ & $\mathop{\mathrm{CMPGA}}\limits_{(N=8)}$ & $\mathop{\mathrm{RMPGA}}\limits_{(N=8)}$ \\
		\midrule
		1  &   & 0.379  & 0.437  & 0.485  & 0.200  & 0.210  \\
		2  &   & 0.370  & 0.418  & 0.466  & 0.187  & 0.205  \\
		3  &   & 0.402  & 0.435  & 0.508  & 0.188  & 0.213  \\
		4  &   & 0.427  & 0.455  & 0.543  & 0.205  & 0.231  \\
		5  &   & 0.455  & 0.464  & 0.574  & 0.227  & 0.255  \\
		6  &   & 0.484  & 0.483  & 0.619  & 0.258  & 0.289  \\
		7  &   & 0.501  & 0.500  & 0.629  & 0.282  & 0.310  \\
		8  &   & 0.527  & 0.517  & 0.657  & 0.316  & 0.335  \\
		9  &   & 0.552  & 0.547  & 0.686  & 0.346  & 0.366  \\
		10 &   & 0.581  & 0.573  & 0.713  & 0.384  & 0.378  \\
		\bottomrule
	\end{tabular}%
	\label{Table: L1dSk-time}%
\end{table}%



\begin{figure}[h]
	\centering
	\caption{Performance of CMPGA and RMPGA with different number of blocks.}
	\subfigure[Epochs versus block number.]{\includegraphics[width=0.49\linewidth]{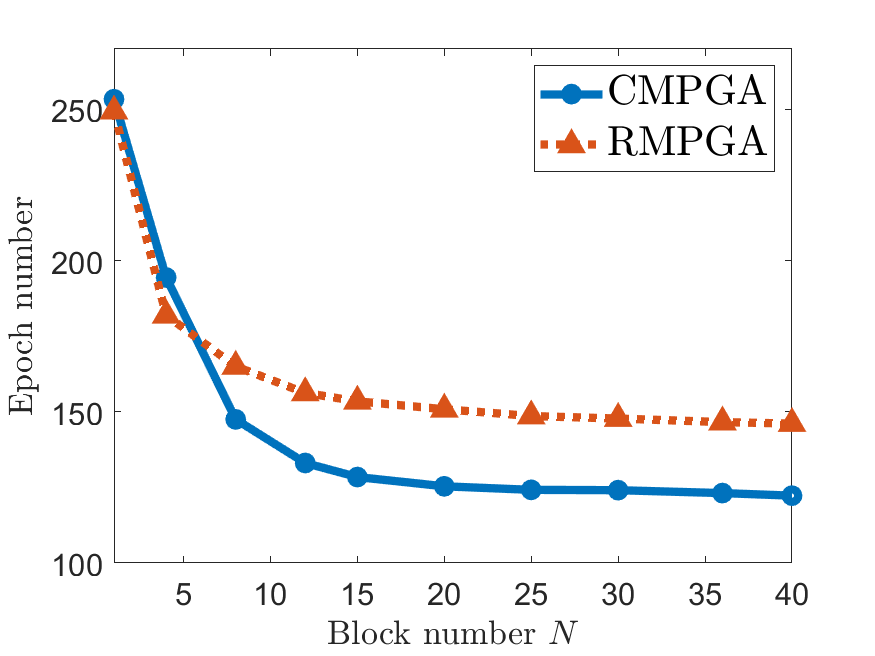}}
	\subfigure[Computational time versus block number.]{\includegraphics[width=0.49\linewidth]{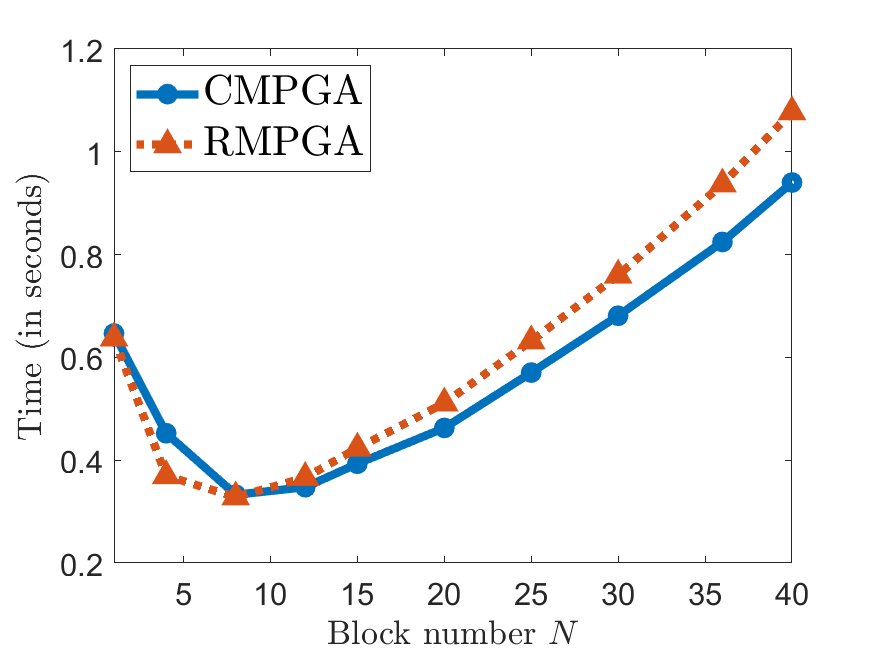}}
	\label{fig:blocksvstime}
\end{figure}

\subsection{$L_1/L_2$ sparse signal recovery problem \eqref{problem:L1dL2}}\label{section: Exp-L1dL2}

This subsection is devoted into numerical experiments on the proposed algorithms for solving problem \eqref{problem:L1dL2}. First, we describe how to construct an instance of problem \eqref{problem:L1dL2} in our test. Given positive integers $m,n$ and a positive number $D$, a matrix $A\in\mathbb{R}^{m\times n}$ is produced by \eqref{def: DCT-Matrix A}. As suggested in \cite{Wang-Yan-Rahimi-Lou:2020IEEE}, the ground truth $x^{\dagger}\in\mathbb{R}^n$ is simulated as an $r$-sparse signal, which has a minimum seperation of at least $2D$ in its support. Besides, we require the dynamic range of $x^{\dagger}$, which is defined by the ratio of the maximum and minimum in the absolute values of its nonzero entries, to be 1000. In our implementations, the nonzero entries of $x^{\dagger}$ is generated by the following MATLAB command:
\begin{equation*}
	\texttt{sign(randn(mu,1)).*10.\^{}(3*rand(mu,1)); }
\end{equation*}
Accordingly, we set $\underline{x} = -1000\times \One_n$ and $\overline{x} = 1000\times \One_n$ in problem \eqref{problem:L1dL2}, where $\One_n$ denotes the $n$-dimensional vector with all entries being 1. Given $A,x^{\dagger},\underline{x},\overline{x}$ and a model parameter $\lambda>0$, we can construct a vector $b\in\mathbb{R}^n$ such that $x^{\dagger}$ is a critical point of problem \eqref{problem:L1dL2} following a similar way to the work of \cite[Section 6.1]{TaoMin:2022SIAM-SC}. Now we obtain a concrete instance of problem \eqref{problem:L1dL2} with $x^{\dagger}\in\mathbb{R}^n$ being one of its critical points.

Next we give the initial point and the stopping criteria for the proposed algorithms. Recall that to satisfy \Cref{Assumption: X is compact}, the initial point $x^{(0)}$ for the proposed algorithms can be selected such that $F(x^{(0)})< 1+\frac{\lambda}{2}\|b\|_2^2$. For $j\in\mathbb{N}_n$, we denote by $e_j \in \mathbb{R}^n$ the one-sparse vector whose $j$-th entry is 1. In view of $-\underline{x} = \overline{x} = 1000\times \One_n$, it is easy to verify that for any $j\in\mathbb{N}_n$ with $a^{\top}_j b \neq 0$, $\mathrm{sign}(a_j^{\top}b)\min\{|a_j^{\top}b|/ \|a_j\|^2_2,\overline{x} \} e_j$ satisfies the aforementioned condition required for $x^{(0)}$. In our tests, we simply find $j_0:=\min \{j\in\mathbb{N}_n: a^{\top}_j b \neq 0 \}$ and set the initial point $x^{(0)} =  \mathrm{sign}(a_{j_0}^{\top}b)\min\{|a_{j_0}^{\top}b|/ \|a_{j_0}\|^2_2,\overline{x} \} e_{j_0}$ for all the proposed algorithms. Furthermore, we terminate the proposed algorithms when 
\begin{equation}\label{eq: L1dL2 Exp-termination}
	\frac{\dist\left(0,\partial Q(x^{(t)},y^{(t)})\right)}{\|(x^{(t)},y^{(t)}) \|_2} < 10^{-7}.
\end{equation}
Below we show that for problem \eqref{problem:L1dL2} the distance $\dist\left(0,\partial Q(x^{(t)},y^{(t)})\right)$ at any $(x,y)\in\dom(Q)$ can be computed in an explicit way. Let $\mathcal{C}$ denote the box $\{x\in\mathbb{R}^{n}: \underline{x}\leq x \leq \overline{x}  \}$. In the case of $\zeta(x)>0$, by \Cref{Corollary: limiting subdiff of frac} \StatementNum{2}, we see that
\begin{equation}\label{eq: z1017 2152}
	\begin{split}
		\dist^2(0,\partial Q(x,y)) = \eta^{-2}(x,y)
		\Big ( &\dist^2(Q(x,y)y-\nabla h(x),\partial \iota_{\mathcal{C}}(x) + \partial\|\cdot\|_1(x))\\
		&+Q^2(x,y)\dist^2(x,\partial(\|\cdot\|_2)^*(y)) \Big ).
	\end{split}
\end{equation}
Note that $\partial \iota_{\mathcal{C}}(x) + \partial\|\cdot\|_1(x)$ is a cuboid in $\mathbb{R}^n$, while $\partial(\|\cdot\|_2)^*(y) = \{0_n \} $ for $\|y\|_2< 1$ and $\partial(\|\cdot\|_2)^*(y) \ \{\beta y:\beta\geq 0\} $ for $\|y\|_2= 1$. Hence, $\dist\left(0,\partial Q(x,y)\right)$ can be directly evaluated via \eqref{eq: z1017 2152} if $\zeta(x)>0$. On the other hand, if $\zeta(x)=0$, it is trivial that $Q$ attains its global minimum at $(x,y)$ and thus $\dist\left(0,\partial Q(x,y)\right) = 0$.

We compare the proposed algorithms with the alternating direction method of multipliers (ADMM) for solving problem \eqref{problem:L1dL2}, which is recently developed in \cite{Wang-Yan-Rahimi-Lou:2020IEEE}. Followig the notations and suggestions in \cite[Section 5]{Wang-Yan-Rahimi-Lou:2020IEEE}, we set $\rho_1=\rho_2=\beta=1$, and choose $\overline{\epsilon} = 10^{-6}$ and $\texttt{j\_max} = 5$ for the inner iterations of ADMM. Besides, the initial point and stopping criteria of ADMM are set as the same as those of the proposed algorithms. In particular, for ADMM, we set $y^{(t)} := \nabla g(x^{(t)}) = \frac{x^{(t)}}{\|x^{(t)}\|_2}$ to verify the stopping condition \eqref{eq: L1dL2 Exp-termination}.

In our experiments, we set $(m,n,r) = (64R,540R,\widetilde{r}R)$ with $\widetilde{r}\in\{6,8,10,12,14  \}$ and $R\in\{8,10,12  \}$. Also, we fix the coherent parameter $D=1$ and the model parameter $\lambda = 2\times 10^{-4}$. For each triple $(m,n,r)$, we generate 50 instances of problem \eqref{problem:L1dL2} as described above. Then we run ADMM, RMPGA with the block number $N=8$, as well as CMPGA with $N=1$ and $8$, for solving the above instances of problem \eqref{problem:L1dL2}. The computational results averaged over the 50 instances with the same $(m,n,r)$ are presented in Tables \ref{Table: L1dL2-Obj} and \ref{Table: L1dL2-time}, which correspond to the objective values and computational time\footnote{The computational time reported does not include the time for verifying the stopping criteria \eqref{eq: L1dL2 Exp-termination}.} (in seconds), respectively. One can observe that the proposed algorithms significantly outperform ADMM in terms of CPU time, while the objective values found by the tested algorithms are comparable. Also, CMPGA with $N=8$ takes about one half of the CPU time needed by the ones with $N=1$ to achieve the same accuracy, which tells that properly exploiting the block separable structure of $f$ can help to improve the computational efficiency. Moreover, we observe that the deterministic algorithm (CMPGA with $N=8$) slightly outperform its randomized counterpart (RMPGA with $N=8$), which is reasonable as a similar observations have been found in the literature, see \cite[Section 3.2]{Amir-Luba:2013SIAM-opt}, for example.

\begin{table}[htbp]
	\centering
	%
	\caption{Results for objective values for solving problem \eqref{problem:L1dL2}.}
	\label{Table: L1dL2-Obj}
	\begin{tabular}{lrrrrr}
		\toprule
		\raisebox{-3pt}[0pt]{$(m,n,r)$} &   ~~~~~~    & \raisebox{-3pt}[0pt]{ADMM}  & $\mathop{\mathrm{CMPGA}}\limits_{\tiny(N=1)}$ & $\mathop{\mathrm{CMPGA}}\limits_{(N=8)}$ & $\mathop{\mathrm{RMPGA}}\limits_{(N=8)}$ \\
		\midrule
		$(512,4320,48)$ &       & 84.35  & 84.35  & 84.35  & 84.35  \\
		$(512,4320,64)$ &       & 68.99  & 68.99  & 68.99  & 68.99  \\
		$(512,4320,80)$ &       & 73.51  & 73.51  & 73.51  & 73.51  \\
		$(512,4320,96)$ &       & 28.35  & 28.35  & 28.35  & 28.35  \\
		$(512,4320,112)$ &       & 24.90  & 24.90  & 24.90  & 24.90  \\
		&       &       &       &       &  \\
		$(640,5400,60)$ &       & 31.66  & 31.66  & 31.66  & 31.66  \\
		$(640,5400,80)$ &       & 36.73  & 36.73  & 36.73  & 36.73  \\
		$(640,5400,100)$ &       & 22.90  & 22.90  & 22.90  & 22.90  \\
		$(640,5400,120)$ &       & 22.53  & 22.53  & 22.53  & 22.53  \\
		$(640,5400,140)$ &       & 47.41  & 47.41  & 47.41  & 47.41  \\
		&       &       &       &       &  \\
		$(768,6480,72)$ &       & 19.15  & 19.15  & 19.15  & 19.15  \\
		$(768,6480,96)$ &       & 19.41  & 19.41  & 19.41  & 19.41  \\
		$(768,6480,120)$ &       & 13.73  & 13.73  & 13.73  & 13.73  \\
		$(768,6480,144)$ &       & 21.00  & 21.00  & 21.00  & 21.00  \\
		$(768,6480,168)$ &       & 24.60  & 24.60  & 24.60  & 24.60  \\
		\bottomrule
	\end{tabular}%
\end{table}%

\begin{table}[htbp]
	\centering
	\caption{Results for computational time for solving problem \eqref{problem:L1dL2}.}
	\label{Table: L1dL2-time}%
	\begin{tabular}{lrrrrr}
		\toprule
		\raisebox{-3pt}[0pt]{$(m,n,r)$} &   ~~~~~~    & \raisebox{-3pt}[0pt]{ADMM}  & $\mathop{\mathrm{CMPGA}}\limits_{\tiny(N=1)}$ & $\mathop{\mathrm{CMPGA}}\limits_{(N=8)}$ & $\mathop{\mathrm{RMPGA}}\limits_{(N=8)}$ \\
		\midrule
		$(512,4320,48)$ &       & 3.94  & 0.16  & 0.12  & 0.12  \\
		$(512,4320,64)$ &       & 4.24  & 0.15  & 0.09  & 0.14  \\
		$(512,4320,80)$ &       & 4.67  & 0.16  & 0.10  & 0.15  \\
		$(512,4320,96)$ &       & 5.64  & 0.21  & 0.11  & 0.19  \\
		$(512,4320,112)$ &       & 6.82  & 0.27  & 0.15  & 0.22  \\
		&       &       &       &       &  \\
		$(640,5400,60)$ &       & 6.58  & 0.23  & 0.13  & 0.17  \\
		$(640,5400,80)$ &       & 7.23  & 0.31  & 0.16  & 0.21  \\
		$(640,5400,100)$ &       & 7.99  & 0.31  & 0.15  & 0.24  \\
		$(640,5400,120)$ &       & 9.48  & 0.37  & 0.16  & 0.29  \\
		$(640,5400,140)$ &       & 11.17  & 0.48  & 0.23  & 0.39  \\
		&       &       &       &       &  \\
		$(768,6480,72)$ &       & 9.90  & 0.32  & 0.16  & 0.24  \\
		$(768,6480,96)$ &       & 10.92  & 0.44  & 0.22  & 0.29  \\
		$(768,6480,120)$ &       & 12.80  & 0.72  & 0.37  & 0.38  \\
		$(768,6480,144)$ &       & 15.00  & 0.63  & 0.28  & 0.57  \\
		$(768,6480,168)$ &       & 18.14  & 0.82  & 0.39  & 0.63  \\
		\bottomrule
	\end{tabular}%
\end{table}%





\backmatter

%
%
%
%
\bmhead{Acknowledgments}

We would like to thank Prof. Lixin Shen for his helpful discussions and valuable suggestions.


\section*{Declarations}


\begin{itemize}
\item Funding: The work of Na Zhang
was supported in part by the National Natural Science Foundation of China grant number
12271181, by the Guangzhou Basic Research Program grant number 202201010426
and by the Basic and Applied Basic Research Foundation of Guangdong Province grant number
2023A1515030046.
The work of Qia Li was supported in part by
the National Natural Science Foundation of China grant 11971499 and the Guangdong
Province Key Laboratory of Computational Science at the Sun Yat-sen University
(2020B1212060032). 
\end{itemize}


%
%
%
%
%

\begin{appendices}
	\section{The proof of Proposition \ref{proposition: KL convergence framework}}
	\label{Appendix: the proof of KL convergence}
	\begin{proof}
		For convenience, we let
		\begin{equation*}
			d_k := \|x^{k+1}-x^k\|_2^2 + a_{k}\|y^{k+1}-y^k\|^2_2
			\text{~~and~~}
			\gamma_k := \phi(H(x^k,y^k)-\xi).
		\end{equation*}
		It follows from Condition \StatementNum{1} that $\{H(x^k,y^k):k\in\mathbb{N} \}$ is non-increasing. This together with Condition \StatementNum{3} implies $H(x^k,y^k) \geq \xi$ for any $k\in\mathbb{N}$. If $H(x^{k_0},y^{k_0}) = \xi$ for some $k_0\in\mathbb{N}$, Condition \StatementNum{1} yields that $H(x^k,y^k) = \xi$ and $d_k = 0$ hold for any $k\geq k_0$. Then, we obtain this proposition from Condition \StatementNum{2} immediately. Hence, we only need to consider the case when
		\begin{equation}\label{eq: z0330 2104}
			H(x^k,y^k) > \xi
		\end{equation}
		holds for any $k\in\mathbb{N}$.
		
		We first show that there exist $K_3\geq 0$ and a continuous concave function $\phi:[0,\epsilon) \to \mathbb{R}_+ $ satisfying Items \StatementNum{1}-\StatementNum{2} in \Cref{Def:KL_property} such that there holds, for $k\geq K_3$,
		\begin{equation}\label{eq: z0330 2101}
			\phi'(H(x^k,y^k)-\xi)\mathrm{~dist}(0,\partial H(x^k,y^k)) \geq 1.
		\end{equation}
		Since $\{(x^k,y^k):k\in\mathbb{N} \}$ is bounded, the set $\Upsilon$ is compact. In view of Condition \StatementNum{3} and \eqref{eq: z0330 2104}, for any $\delta>0$ and $\epsilon>0$, there exists $K_3>0$ such that
		\begin{displaymath}
			(x^k,y^k) \in \{(x,y)\in\mathbb{R}^n\times \mathbb{R}^m: \dist((x,y),\Upsilon)<\delta,~\xi<H(x,y)<\xi+\epsilon \}
		\end{displaymath}
		holds for any $k\geq K_3$. Hence, by invoking \Cref{lemma: Uniformized KL property} \StatementNum{1}, we deduce that \eqref{eq: z0330 2101} holds for any $k\geq K_3$.
		
		Then, for $k\geq K := \max\{K_1,K_2,K_3 \}$, we have
		\begin{align}
			\sqrt{d_k} \notag
			&\overset{(\uppercase\expandafter{\romannumeral 1})}{\leq}
			\sqrt{\frac{1}{C_1}(H(x^k,y^k)-H(x^{k+1},y^{k+1}))}\\
			&\overset{(\uppercase\expandafter{\romannumeral 2})}{\leq}  \sqrt{\frac{1}{C_1}(\gamma_k-\gamma_{k+1})/\phi'(H(x^k,y^k)-\xi )}\notag \\
			&\overset{(\uppercase\expandafter{\romannumeral 3})}{\leq}  \sqrt{\frac{2C_2}{C_1}(\gamma_k-\gamma_{k+1})\frac{1}{2}\left( \|x^k-x^{k-1}\|_2 + \sqrt{a_{k-1}} \|y^k-y^{k-1}\|_2 \right)  } \notag\\
			&\overset{(\uppercase\expandafter{\romannumeral 4})}{\leq}  \frac{C_2}{C_1}(\gamma_k-\gamma_{k+1}) +  \frac{1}{4}\left( \|x^k-x^{k-1}\|_2 + \sqrt{a_{k-1}}\|y^k-y^{k-1}\|_2 \right)  \notag\\
			&= \frac{C_2}{C_1}(\gamma_k-\gamma_{k+1}) +  \frac{1}{4}\sqrt{ d_{k-1}+2\sqrt{a_{k-1}}\|x^k-x^{k-1}\|_2\|y^k-y^{k-1}\|_2 }\notag\\
			&\overset{(\uppercase\expandafter{\romannumeral 5})}{\leq}  \frac{C_2}{C_1}(\gamma_k-\gamma_{k+1}) +  \frac{1}{4}\sqrt{ 2d_{k-1} } 
			\leq\frac{C_2}{C_1}(\gamma_k-\gamma_{k+1}) +  \frac{1}{2}\sqrt{ d_{k-1} },\label{eq: z1016 511}
		\end{align}
		where the inequality $(\uppercase\expandafter{\romannumeral 1})$ follows from Condition \StatementNum{1}; the inequality $(\uppercase\expandafter{\romannumeral 2})$ follows from the concavity of $\phi$; the inequality $(\uppercase\expandafter{\romannumeral 3})$ follows from \eqref{eq: z0330 2101} and Condition \StatementNum{2}; the inequalities $(\uppercase\expandafter{\romannumeral 4})$ and $(\uppercase\expandafter{\romannumeral 5})$ follow from the facts that $\sqrt{\alpha\beta}\leq \frac{\alpha+\beta}{2}$ and $2\alpha\beta\leq\alpha^2+\beta^2$ hold for any $\alpha,\beta>0$, respectively. By summing \eqref{eq: z1016 511} from $k=K+1$ to $k=J>K+1$, we have 
		\begin{equation*}
			\sum_{k=K+1}^{J}\sqrt{d_k} \leq \frac{C_2}{C_1}(\gamma_{K+1}-\gamma_{J+1}) + \frac{1}{2}\sum_{k=K}^{J-1}\sqrt{d_k}.
		\end{equation*}
		It follows that
		\begin{equation}\label{eq: z1016 2205}
			\frac{1}{2}\sum_{k=K+1}^{J-1}\sqrt{d_k} \leq \frac{C_2}{C_1}(\gamma_{K+1}-\gamma_{J+1})-\sqrt{d_J}+\frac{1}{2}\sqrt{d_K}\leq \frac{C_2}{C_1}\gamma_{K+1}+\frac{1}{2}\sqrt{d_K}.
		\end{equation}
		By passing to the limit with $J\to\infty$ on the both sides of \eqref{eq: z1016 2205}, we obtain $\sum_{k=0}^{+\infty}\sqrt{d_k} < +\infty$ and Item \StatementNum{1}. Item \StatementNum{2} is the direct result of Item \StatementNum{1}, and Item \StatementNum{3} follows from Item \StatementNum{1} and Condition \StatementNum{2}.  
	\end{proof}
\end{appendices}






\end{document}